\documentclass[12pt,twoside]{amsart}

\usepackage{times,a4wide,amssymb,amsmath,amsthm,color}
\usepackage[draft=false]{hyperref}
\usepackage{amsrefs}

\theoremstyle{plain}
\newtheorem{thm}{Theorem}[section]
\newtheorem{prop}[thm]{Proposition}
\newtheorem{cor}[thm]{Corollary}
\newtheorem{lemma}[thm]{Lemma}
\newtheorem{claim}[thm]{Claim}
\theoremstyle{remark}
\newtheorem*{rem}{Remark}

\theoremstyle{definition}
\newtheorem{defn}[thm]{Definition}
\newtheorem{example}[thm]{Example}

\newcommand{\C}{{\mathbb C}}
\newcommand{\R}{{\mathbb R}}
\newcommand{\T}{{\mathbb T}}
\newcommand{\Z}{{\mathbb Z}}
\newcommand{\E}{{\mathbb E}}
\newcommand{\Pro}{{\mathbb P}}

\newcommand{\Ee}{{\mathbf E}}
\newcommand{\Prr}{{\mathbf P}}

\DeclareMathOperator{\re}{Re}
\DeclareMathOperator{\im}{Im}

\newcommand{\F}{\mathcal{F}}
\newcommand{\K}{\mathcal{K}}
\newcommand{\cE}{\mathcal{E}}
\newcommand{\cN}{\mathcal{N}}
\newcommand{\Sone}{\mathcal{S}^1}
\newcommand{\cA}{{\mathcal{A}}}

\newcommand{\Abs}[1]{\left|#1\right|}

\DeclareMathOperator{\sgn}{sgn}
\DeclareMathOperator{\mes}{meas}

\newenvironment{rlist}
{

\begin{enumerate}}
{\end{enumerate}}

\title{On the number of nodal domains of toral eigenfunctions}
\thanks{The research leading to these results has received funding from the European Research Council under the European Union's Seventh Framework Programme (FP7/2007-2013), ERC grant agreement n$^{\text{o}}$ 335141}
\author{Jeremiah Buckley}
\author{Igor Wigman}
\address{Dept. of Mathematics, King's College London, Strand, London, WC2R 2LS, United Kingdom.}
\email{jeremiah.buckley@kcl.ac.uk}
\email{igor.wigman@kcl.ac.uk}
\begin{document}
\begin{abstract}
We study the number of nodal domains of toral Laplace eigenfunctions. Following Nazarov-Sodin's results for random fields and Bourgain's de-randomisation procedure we establish a precise asymptotic result for ``generic" eigenfunctions. Our main results in particular imply an optimal lower bound for the number of nodal domains of generic toral eigenfunctions.
\end{abstract}
\maketitle

\section{Introduction and main results}

\subsection{Toral eigenfunctions}
Let $\T^{2}=\R^{2}/\Z^{2}$ be the standard $2$-dimensional torus and let $\Delta$ be the Laplacian on $\T^{2}$. We are interested
in the eigenfunctions of $\Delta$, i.e., functions $f:\T^{2}\rightarrow \R$ satisfying the Schr\"{o}dinger equation
\begin{equation}
\label{eq:Schrodinger}
\Delta f + 4 \pi^2 Ef=0
\end{equation}
for some $E\ge 0$. It is well known that the spectrum of $\Delta$ is purely discrete; $E$ lies in the set
\begin{equation}
\label{eq:S def}
S= \{a^{2}+b^{2}:\: a,b\in \Z \}
\end{equation}
of all integer numbers  expressible as the sum of two integer squares. For such an integer $E\in S$ let
\begin{equation*}
\cE=\cE_{E}=\{\xi\in\Z^2:|\xi|^2=E\}
\end{equation*}
be the collection of lattice points lying on the circle in $\R^{2}$ centred at the origin and of radius $\sqrt{E}$, and let
\begin{equation}
\label{eq:N def}
N=N_{E}=|\cE_{E}|
\end{equation}
be its size.

Given $E\in S$ we may express the general (complex-valued) solution of \eqref{eq:Schrodinger} as
\begin{equation}
\label{eq:fa def}
f(x)=f_{a}(x)=f_{E;a}(x) = \sum\limits_{\xi\in\cE}a_{\xi}e(\langle \xi,x\rangle),
\end{equation}
where $a=(a_{\xi})_{\xi\in\cE}$ are some complex coefficients, $x=(x_{1},x_{2})\in\T^{2}$, $e(t)=e^{2\pi i t}$
and $\langle\cdot,\cdot\rangle$ is the usual inner product on $\R^2$.
Assuming that
\begin{equation}
\label{eq:a(-xi)=bar(a(xi))}
a_{-\xi}=\overline{a_\xi}
\end{equation}
will guarantee that $f$ is real valued; since our primary focus
is the zero set of $f$, we may also normalize $(a_{\xi})$ to satisfy
\begin{equation}
\label{eq:sum axi^2=1}
\sum_{\xi\in\cE} |a_\xi|^2=1.
\end{equation}
We will assume throughout this article that the coefficients satisfy both \eqref{eq:a(-xi)=bar(a(xi))} and \eqref{eq:sum axi^2=1}.

\subsection{Nodal domains}\label{sec:nod dom into}
The {\em nodal components} of a real-valued smooth function $\phi:\T^2\rightarrow\R$ are the connected components of its zero set $\mathcal Z(\phi)=\phi^{-1}(0)$ (called the ``nodal set"), and the {\em nodal domains} are the connected components of its complement $\T^2\setminus \phi^{-1}(0)$. Our primary focus is on the number $\cN_{f_{a}}$ of nodal domains of the function $f_{a}$ given by \eqref{eq:fa def} (recall that our assumption \eqref{eq:a(-xi)=bar(a(xi))} guarantees that $f_{a}$ is real-valued); typically --- for example, if no point of the nodal set is a critical point --- $\cN_{f_{a}}$ is (almost) equal to the number of nodal components of $f_{a}$ --- they may differ by at most $1$. The Courant Nodal Theorem (valid in a much more general scenario) implies that
\begin{equation*}
  \cN_{f_{a}} = O(E)
\end{equation*}
with the constant involved in the `O'-notation absolute and explicit. Bourgain ~\cite{B}*{Proposition 1, Theorem 2} showed that if all of the coefficients in \eqref{eq:fa def} are equal,
\begin{equation*}
  f(x)=\frac1{\sqrt N} \sum\limits_{\xi\in\cE}e(\langle \xi,x\rangle),
\end{equation*}
then, for generic values of $E$, the number of nodal domains satisfies the asymptotic law
\begin{equation}\label{eq:Nf a=1 c0 E}
  \cN_{f}\sim c_{0}\cdot E,\qquad E\to\infty
\end{equation}
where $c_{0}>0$ is some positive constant borrowed from theory of random fields (``the universal Nazarov-Sodin constant" \cite{NaSo2}, see Section~\ref{sec:NS constant} below).

Concerning a lower bound for $\cN_{f_{a}}$, there exist Laplace eigenfunctions of arbitrarily large energy $E$ with only $2$ nodal domains (at least on the square); hence there is no nontrivial lower bound for their number (this result goes back to A. Stern \cite{St} although we refer the reader to \cite{BeHe}*{Theorem 4.1} and the discussion that follows it; the analogous result on the sphere is given by \cite{Lew}*{Theorems 1 and 2}). Moreover, there exist \cite{KW}*{Proposition 3.2} sequences $\{E\}\subseteq S$ of energy levels with only $o(E)$ nodal domains as $E\to\infty$ for ``most'' coefficients $(a_{\xi})_{\xi\in\cE}$. However, it is widely believed that such a situation is impossible in some generic scenario, and it is desirable to show a lower bound of the ``correct" order of magnitude
\begin{equation}\label{eq:Nf>>E}
  \cN_{f_{a}} \gtrsim E
\end{equation}
holding uniformly for generic sequences of energies $\{E\}$ and
coefficients $(a_{\xi})$ satisfying some mild extra
assumptions\footnote{We thank Jean Bourgain for raising this question.};
\eqref{eq:Nf>>E} will follow as a straightforward corollary of our main results (cf. Corollary \ref{cor:lower bound}).

\subsection{The Nazarov-Sodin constant}

\label{sec:NS constant}

Given a symmetric probability measure $\mu$ on the unit circle $\Sone \subseteq \R^{2} $ we may consider the (random) Gaussian monochromatic wave $h_{\mu}:\R^{2}\rightarrow \R$ with unit wavenumber and directions distributed according to $\mu$. (Here symmetric means that for any arc $I$ we have $\mu(I)=\mu(-I)$, which guarantees that the field $h_\mu$ is real-valued.) The random field $h_{\mu}$ is uniquely defined as the centred stationary (i.e. the law of $h_{\mu}$ is invariant under translations) Gaussian random field, whose covariance function
\begin{equation*}
\E[h_{\mu}(x)h_{\mu}(x')] = \E[h_{\mu}(x-x')h_{\mu}(0)] =  \int\limits_{\Sone}e(\langle x-x',\theta \rangle)d\mu(\theta)
\end{equation*}
equals the Fourier transform of $\mu$ viewed as a measure on $\R^2$ (supported on $\Sone$). (Equivalently, $\mu$ is the {\em spectral measure} of $h_{\mu}$.)

Now for $R>0$ let $\cN_{h_{\mu}}(R)$ be the number of nodal domains of $h_{\mu}$ lying entirely inside the disc $B(R)$ centred at the origin of radius $R$. Nazarov and Sodin \cite{NaSo2}*{Theorem 1} proved that there exists a non-negative number $$c_{\mu}=c_{NS}(\mu)\ge 0$$ such that as $R\rightarrow\infty$ the expected number of nodal domains lying in $B(R)$ satisfies the asymptotic
\begin{equation}
\label{eq:Ngmu(R)<->c piR^2}
\E[\cN_{h_{\mu}}(R)]= c_{\mu}\cdot \pi R^{2}(1+o(1)),
\end{equation}
and gave some very mild conditions on $\mu$ that guarantee that $c_{\mu}$ is strictly positive; they also proved the almost sure convergence and convergence in mean of
\begin{equation*}
  \frac{\cN_{h_{\mu}}(R)}{\pi R^{2}}
\end{equation*}
to $c_{\mu}$, under some extra regularity assumptions on $\mu$. The constant $c_0$ mentioned in Section~\ref{sec:nod dom into}, the ``universal Nazarov-Sodin constant'', is given by $c_0=c_{NS}(\mu_{\mathrm{unif}})$ where
\begin{equation*}
  d\mu_{\mathrm{unif}}(\theta)=\frac{d\theta}{2\pi}
\end{equation*}
is the uniform measure on $\Sone$.

We may also consider $c_{NS}$ as a map
\begin{equation*}
  c_{NS}:\mathcal P(\Sone)\rightarrow \R_{+}
\end{equation*}
from the set of symmetric probability measures on the unit circle to the non-negative real numbers. It is known \cite{KW}*{Theorem 3.1, Corollary 3.3} that $c_{NS}$ is a continuous map with respect to the topology of weak convergence on $\mathcal P(\Sone)$, and moreover that $c_{NS}$ attains precisely an interval
$$c_{NS}(\mathcal{P}(\Sone)) = [0,c_{1}]$$
with some\footnotemark maximal value $0<c_{1}<\infty$.

\footnotetext{It was conjectured \cite{KW}*{Conjecture 3.4} that $c_{1}=c_{0}$ is the universal constant from \eqref{eq:Nf a=1 c0 E},
also known to be valid for Berry's Random Wave Model (equivalently random spherical harmonics \cite{NaSo}).}

Returning to $f_{a}$ as given by \eqref{eq:fa def}, we may view each $e(\langle\xi,x\rangle)$ as representing a plane wave (projected onto the torus) and think of the coefficient $a_{\xi}$
as amplifying the component propagating in the direction $\xi$. Hence the (deterministic) function $f_{a}$ could be viewed as a monochromatic superposition of plane waves of wavenumber $\sqrt{E}$
(projected onto the torus) with directions distributed according to the probability measure
\begin{equation*}
\mu_{E,a}=\sum_{\xi\in\cE} |a_\xi|^2 \delta_{\xi/\sqrt{E}}
\end{equation*}
on the unit circle $\Sone \subseteq \R^{2} $.
It is then reasonable to compare the (deterministic) number $\cN_{f_{a}}$ of
nodal domains of $f_{a}$ on the torus to the number of nodal domains
of $h_{\mu}$ with $\mu=\mu_{E,a}$ lying in a square in $\R^{2}$ of side-length $\sqrt{E}$, in the high energy limit $E\rightarrow\infty$;
the latter of these two is asymptotic to $c_{NS}(\mu)\cdot E$ (for example, in expectation, cf. \eqref{eq:Ngmu(R)<->c piR^2}).
This is precisely the main concern of our principal results, under
some ``generic" restrictions on the coefficients $(a_{\xi})$ and sequences of energy levels $\{E_{j}\}$
(see Theorems~\ref{thm: main thm uniform} and \ref{thm: main thm mu E cgt} below).

\subsection{Statement of the main results}
We will make some assumptions on the values of $E$ and the coefficients $a_\xi$ that we allow. First we make an assumption on the number of additive relations in the set $\cE$.
\begin{defn}[Constraint on the energy levels]\textcolor{white}{h}
  \begin{enumerate}
    \item  We say that a set of distinct $$\xi_1,\dots,\xi_l\in\cE$$ is \emph{minimally vanishing} if
        \begin{equation}\label{eq: l vanishing freq}
          \xi_1+\cdots+\xi_l=0
        \end{equation}
        and no proper sub-sum of \eqref{eq: l vanishing freq} vanishes.
    \item We say that $\cE$ satisfies the condition $I(\gamma,B)$ for $0<\gamma<\frac12$ and $B\geq1$ if, for all $3\leq l\leq B$, the number of minimally vanishing subsets of $\cE$ of length $l$ is at most $N^{\gamma l}$ (recall $N$ as in \eqref{eq:N def}).
  \end{enumerate}
\end{defn}

We will assume that $\cE$ satisfies $I(\gamma,B)$ for sufficiently large $B$; this assumption is valid \citelist{\cite{B-B}*{Theorem 17} \cite{B}*{Lemma 4}} for a density $1$ sequence $\{E\}\subseteq S$ (recall \eqref{eq:S def}). Further we will restrict the coefficients $a_\xi$ that we consider.

\begin{defn}[Constraint on the coefficients]\textcolor{white}{h}
  \begin{enumerate}
    \item We say that a function $g:\R\rightarrow \R_{>0}$ is a \emph{slowly growing function} if for all $\delta>0$ we have $$ g(x) =o(x^{\delta}), \qquad x\to\infty.$$
    \item Recalling that the coefficients $a=(a_{\xi})_{\xi\in\cE}$ are normalized by \eqref{eq:sum axi^2=1}, we let
    \begin{equation}\label{eq: M def}
      M=M(E,a)=\max\{|a_\xi|^2:\xi\in\cE\},
    \end{equation}
    and we further recall that $N$ is given by \eqref{eq:N def}. Given a slowly growing function $g$ we say that the coefficients $a=(a_{\xi})_{\xi\in\cE}$ are \emph{of class} $\cA(g)$ if they satisfy
    \begin{equation}\label{eq: bound on coeff}
      M\leq \frac{g(N)}{N}.
    \end{equation}
  \end{enumerate}
\end{defn}

Examples of slowly growing functions are given by any bounded $g$, the function $g(x)=\log_+ x$, or indeed any power of $\log$. At first glance \eqref{eq: bound on coeff} may seem a restrictive condition, but in fact it is generic with respect to the natural measure on the sphere. Specifically, if we normalise so that the sphere has measure $1$, the measure of the set of points $a=(a_{\xi})_{\xi\in\cE}$ on the $N$-sphere (that is, satisfying \eqref{eq:sum axi^2=1}) that also satisfy \eqref{eq: bound on coeff} is asymptotically $1$ for large $N$, if we chose $g$ to be a large enough power of $\log$. This follows immediately from L\'{e}vy's concentration of measure on the sphere.

Our first principal result is the following theorem:
\begin{thm}\label{thm: main thm uniform}
  Fix a slowly growing function $g$ and let $\{E \}\subseteq S$ be a sequence of energy levels, $E\rightarrow\infty$, such that $N_{E}\to\infty$ as $E\to\infty$, and that there exists $B(E)\to\infty$ and $\gamma<\frac12$ such that $\cE$ satisfies $I(\gamma,B(E))$ for all $E$. Then, given $\epsilon>0$, there exists $E_0=E_0(\epsilon,\gamma,g)$ such that for all $E\ge E_{0}$
  \begin{equation}\label{eq:|Nf/E-cNS|<eps}
    \Abs{\frac{\cN_{f_{a}}}{E}-c_{NS}(\mu_{E,a})}<\epsilon
  \end{equation}
  uniformly for all coefficients $(a_\xi)$ of class $\cA(g)$.
\end{thm}

Note that if $c_{NS}(\mu_{E,a})$ in \eqref{eq:|Nf/E-cNS|<eps} is bounded away from $0$, then \eqref{eq:|Nf/E-cNS|<eps} implies a uniform lower bound of the type $$\cN_{f_{a}}\gtrsim E,$$ raised in Section~\ref{sec:nod dom into}. Nazarov and Sodin \cite{NaSo2} showed that it is quite easy to check whether for a given symmetric measure $\mu\in \mathcal P(\Sone)$ the corresponding Nazarov-Sodin constant $c_{NS}(\mu)=0$, and hence whether for a given set $K\subseteq \mathcal P(\Sone)$ the image $c_{NS}(K)$ is bounded away from zero.

\begin{cor}[Generic uniform lower bound for nodal domains number]
\label{cor:lower bound}
Let $g$ be a slowly growing function, and $K$ be a closed (i.e. compact) subset $K\subseteq\mathcal P(\Sone)$ of the set of probability measures on $\Sone$ such that
\begin{equation}\label{eq: cNS bdd from 0}
  \min_{\mu\in K}c_{NS}(\mu)>0.
\end{equation}
Suppose that $\{E \}\subseteq S$ is a sequence of energy levels, $E\rightarrow\infty$, such that $N_{E}\to\infty$, and that there exists $B(E)\to\infty$ such that $\cE$ satisfies $I(\gamma,B(E))$ for all $E$. Then there exist $c=c(K)>0$ and $E_0=E_0(g,K)$ such that for all $E\geq E_0$ and for all coefficients $(a_\xi)$ of class $\cA(g)$ satisfying $$\mu_{E,a}\in K$$ we have the lower bound
\begin{equation*}
\cN_{f_{a}} >c \cdot E.
\end{equation*}
\end{cor}

It is clear that Corollary \ref{cor:lower bound} is {\em optimal} in the sense that if we remove the hypothesis \eqref{eq: cNS bdd from 0} then no such lower bound exists.

\begin{example}
  An interesting example, where we can describe precisely the sets $K$ that satisfy \eqref{eq: cNS bdd from 0}, is given by considering the measures that have some extra symmetry. Specifically, let $\mathcal P_0(\Sone)$ denote the subset of $\mathcal P(\Sone)$ that consists of measures that are invariant with respect to rotation by $\pi/2$ and complex conjugation (that is, the map $(\zeta_1,\zeta_2)\mapsto (\zeta_1,-\zeta_2)$). For instance, if we choose all of the coefficients $a_\xi$ to be equal in modulus then $\mu_{E,a}\in \mathcal P_0(\Sone)$. More generally, if $|a_\xi|$ is constant for $\xi\in\{(\pm \xi_1,\pm \xi_2),(\pm \xi_2,\pm \xi_1)\}$ then $\mu_{E,a}\in \mathcal P_0(\Sone)$. By \cite{KW}*{Theorem 2.1} there are exactly two measures in $\mathcal P_0(\Sone)$ with vanishing Nazarov-Sodin constant, the Cilleruelo measure
  \begin{equation}
    \nu_0=\frac14(\delta_1+\delta_i+\delta_{-1}+\delta_{-i})
  \end{equation}
  and the tilted Cilleruelo measure $\nu_1$, which is obtained by rotating the Cilleruelo measure by $\pi/4$. Moreover $\nu_0$ (respectively $\nu_1$) is the only measure in $\mathcal P_0(\Sone)$ whose fourth Fourier coefficient is $1$ (respectively $-1$). Thus, for a closed subset of $K\subset\mathcal P_0(\Sone)$ we see that \eqref{eq: cNS bdd from 0} is equivalent to the condition
  \begin{equation*}
    \min_{\mu\in K}(1-|\widehat{\mu}(4)|)>0.
  \end{equation*}
\end{example}

Since $\mathcal P(\Sone)$ is compact with respect to the topology of weak convergence, and the functional $c_{NS}$ is continuous, Theorem~\ref{thm: main thm uniform} may be restated as follows (we write $\Rightarrow$ to indicate weak convergence of probability measures).

\begin{thm}\label{thm: main thm mu E cgt}
Let $\{E\}\subseteq S$ be a sequence of energies such that $N_{E}\to\infty$ as $E\rightarrow\infty$, and there exists $B(E)\to\infty$ such that $\cE$ satisfies $I(\gamma,B(E))$. Let $g$ be a slowly growing function, and $a_{E}=(a_{E,\xi})_{\xi\in\cE_{E}}$ a sequence of coefficients of class $\cA(g)$ such that
$$\mu_{E,a_{E}}\Rightarrow\mu.$$
Then as $E\rightarrow \infty$
\begin{equation*}
\cN_{f}= c_{NS}(\mu)\cdot E(1+o(1)),
\end{equation*}
where $f=f_{E;a_{E}}$.
\end{thm}

\subsection{Acknowledgments}

J. Bourgain first piqued our interest in finding deterministic conditions that guarantee \eqref{eq:Nf>>E}, we thank him for this and for interesting discussions. We had many useful conversations with M. Sodin, who contributed many ideas, in particular a simplification of the proof of Theorem \ref{thm: main thm uniform}. We are grateful to Z. Rudnick and P. Sarnak for their interest in our work and to B. Helffer, M. Krishnapur and I. Polterovich for stimulating discussions.

\section{Outline of the paper}
\subsection{The main ideas}
Here we sketch the main ideas and the structure of the proof of the two main theorems. The first step in our approach is to see that the diameter of ``most'' of the nodal domains of $f$ is at most of order $\frac 1{\sqrt E}$. We introduce a large parameter $1\ll R\ll\sqrt E$ and show that
\begin{equation*}
  \cN_f=\frac E{R^2}\int_{\T^2}\cN_f\left(x,\frac R{\sqrt{E}}\right)\,dx+\langle\text{small error}\rangle,
\end{equation*}
where $\cN_f(x,\epsilon)$ denotes the number of nodal domains of $f$ contained in the open box centred at $x$ of sidelength $\epsilon$; the precise statement is given by Lemma~\ref{lem: localisation} whose proof is given in Section~\ref{sec:Lemma local}. This allows us to study the nodal domains ``locally''. We then fix a point $x\in\T^2$, and ``blow-up'' the function $f$ in a small neighbourhood of $x$. Specifically we define
\begin{equation*}
  F_x(y)=f\left(x+\frac R{\sqrt{E}}y\right),
\end{equation*}
and a key step in our approach is to apply a de-randomisation technique, due to Bourgain, to see that $(F_x)_{x\in\T^2}$ approximates a Gaussian field, when we think of $\T^2$ as a probability space.

To this end we introduce another large parameter $K$, divide the unit circle into arcs of length $\frac 1{2K}$, denote the centre of each arc by $\zeta^{(k)}$ and denote by $\cE^{(k)}$ the elements of $\cE$ whose arguments are close to that of $\zeta^{(k)}$. Notice that
\begin{equation*}
  e\left(\left\langle\xi,x+\frac R{\sqrt E}y\right\rangle\right)=e\left(\left\langle\xi,x\right\rangle\right)e\left(\left\langle\frac \xi{\sqrt E},Ry\right\rangle\right)\approx e\left(\left\langle\xi,x\right\rangle\right)e\left(\left\langle \zeta^{(k)},Ry\right\rangle\right)
\end{equation*}
for $\xi \in \cE^{(k)}$. Here the approximation is in the $\mathcal C^1$ norm, and we emphasise that this expression allows us to ``separate'' the dependence on $x$ and $y$, that the expression involving $y$ depends only on $k$, and not on the value of $\xi \in \cE^{(k)}$, and that the number of terms $\xi \in \cE^{(k)}$ for which this holds is large. This allows us to show that\footnote{To simplify matters, the definition of $\phi_x$ and $b_k$ in this section differs slightly from their ``true definitions'' in the next section.} for ``most'' $x$
\begin{align*}
  F_x(y)=\sum_{k}\sum_{\xi \in \cE^{(k)}}a_\xi e\left(\left\langle\xi,x+\frac R{\sqrt E}y\right\rangle\right) &\approx \sum_{k} \left(\sum_{\xi\in\cE^{(k)}} a_\xi e(\langle \xi,x\rangle)\right) e(\langle R\zeta^{(k)},y\rangle)\\
  &= \sum_{k} b_k(x) e(\langle R\zeta^{(k)},y\rangle) = \phi_x(y)
\end{align*}
where the approximation is in the $\mathcal C^1$ norm and
\begin{equation*}
  b_k(x) = \sum_{\xi\in\cE^{(k)}} a_\xi e(\langle \xi,x\rangle),
\end{equation*}
the precise statement is given in Proposition~\ref{prop: gathering} and proved in Section~\ref{sec:approx F by phi}. We then show that these coefficients $b_k(x)$, when we allow $x$ to vary over $\T^2$, approximate a sequence of independent normal random variables in distribution, which allows us to see that\footnote{There is also difference in the definition of $c_k$ from that of the next section, which corresponds to the change in $b_k$.}
\begin{equation*}
  \phi_x(y)\approx\sum_{k} c_k e(\langle R\zeta^{(k)},y\rangle) = \Phi(y),
\end{equation*}
where once more the approximation is in the $\mathcal C^1$ norm and $c_k$ are independent normal random variables with an appropriately chosen variance. This is the content of Proposition~\ref{prop: pass to random} which we prove in Section~\ref{sec:derandom}.

Let us outline, heuristically, why we might expect the sequence $b_k$ to be approximately normal, and where the assumptions we impose on the energy levels and the coefficients arise. Intuitively, we may regard each $b_k$ as a sum of `random variables' (thinking of $\T^2$ as a probability space), and so under some reasonable assumptions expect to see Gaussianity, via the CLT, if the size of each $\cE^{(k)}$ is growing. Clearly it is reasonable to assume that no single coefficient $a_\xi$ is too large (in the extreme case, where one single coefficient has modulus $1$ and the remainder are $0$, we do not have any Gaussianity), and this is guaranteed by condition \eqref{eq: bound on coeff} --- in fact the hypothesis \eqref{eq: bound on coeff} is much stronger. Furthermore, it is reasonable to see convergence in distribution to a normal by convergence of moments, and computing the (joint) moments of the sequence $b_k$ immediately leads one to consider vanishing sums of elements of $\cE$. It is in order to control these moments that we impose the hypothesis $I(\gamma,B)$. Both of these hypotheses are used in the proof of Lemma~\ref{lem: eps Gaussian}, which we consider to be the heart of the proof.

Finally we use the work of Nazarov and Sodin to compute asymptotically the number of nodal domains of the approximating Gaussian field. First, in Proposition~\ref{prop: Psi negligible}, we show that the `small' (in the $\mathcal C^1$ norm) errors we made in passing from $F_x$ to $\phi_x$ and from $\phi$ to $\Phi$ do not contribute significantly to the nodal count, by using techniques developed by Nazarov and Sodin. Then, in Proposition~\ref{prop: NS constant}, we use a theorem of Nazarov and Sodin to compute the asymptotic growth of the number of nodal domains of $\Phi$. The proofs of these propositions are to be found in Section~\ref{sec: Naz Sod}.

\subsection{Outline of the paper}
The paper is organised as follows. In Section~\ref{sec:Proofs and notn} we introduce the main notation that will be used throughout the paper, we state Lemma~\ref{lem: localisation} and Propositions~\ref{prop: gathering}-\ref{prop: NS constant} and we prove Theorems~\ref{thm: main thm uniform} and \ref{thm: main thm mu E cgt} assuming the lemma and propositions. The remainder of the paper is dedicated to proving these ingredients.

Lemma~\ref{lem: localisation}, which allows us to ``localise'' the problem, is proved in Section~\ref{sec:Lemma local}. In Section~\ref{sec:approx F by phi} we prove Proposition~\ref{prop: gathering}, which allows us to pass from studying the nodal domains of $F_x$ to $\phi_x$. In Section~\ref{sec:derandom} we apply Bourgain's derandomisation technique to prove Proposition~\ref{prop: pass to random}, which connects the nodal count of the deterministic function $\phi$ to that of the Gaussain field $\Phi$. Finally we prove Propositions~\ref{prop: Psi negligible} and \ref{prop: NS constant} in Section~\ref{sec: Naz Sod}, where we apply Nazarov and Sodin's results.

Finally some notation: We use $c$ or $C$ to denote absolute constants, that do not depend on any of the parameters, but whose value may change from one occurrence to the next. We denote by $C(B)$ (we choose $B$ for purely demonstrative purposes, in principle it may be any parameter) quantities that we are thinking of as being constant for `local purposes' but that depend on the parameter $B$. We write $f\lesssim g$ if there exists an absolute constant $C$ such that $f\leq Cg$. We write $f=O(g)$ if $|f|\lesssim g$. We write $f=O_B(g)$ if $|f|\leq C(B) g$.

\section{Proofs of Theorems~\ref{thm: main thm uniform} and \ref{thm: main thm mu E cgt}}\label{sec:Proofs and notn}

\subsection{Theorem \ref{thm: main thm mu E cgt} implies Theorem \ref{thm: main thm uniform}}
We show that the statements of the two main theorems are equivalent. It is clear that Theorem~\ref{thm: main thm uniform} implies Theorem~\ref{thm: main thm mu E cgt}. To see the converse implication, suppose that Theorem~\ref{thm: main thm mu E cgt} holds, but that Theorem~\ref{thm: main thm uniform} fails. Then there exists $\epsilon>0$ and a sequence $E_j,a_j$ with $E_j\to\infty$ satisfying the hypotheses with
\begin{equation}\label{eq:contradiction}
  \Abs{\frac{\cN_f}{E_j}-c_{NS}(\mu_{E_j,a_j})}\geq\epsilon
\end{equation}
for all $j$. By passing to a subsequence, we may assume that $$\mu_{E_j,a_j}\Rightarrow\mu$$ for some $\mu\in\mathcal P(\Sone)$, implying that $$c_{NS}(\mu_{E_j,a_j})\to c_{NS}(\mu).$$ Furthermore, by Theorem~\ref{thm: main thm mu E cgt} we have $\frac{\cN_f}{E_j}\to c_{NS}(\mu)$, which contradicts \eqref{eq:contradiction}. The rest of this paper is dedicated to proving Theorem~\ref{thm: main thm mu E cgt}.

\subsection{Main notation: considering eigenfunctions locally}

Recall that the main object of our study $f$ is defined by \eqref{eq:fa def}. Choose a large parameter $R>0$ and write
\begin{equation}\label{eq:Fx def}
  F_x(y)=f\left(x+\frac R{\sqrt{E}}y\right)
\end{equation}
for $y\in\left(-1,1\right)^2$. We will be mainly interested in $y\in\left(-\frac12,\frac12\right)^2$, and for any function $F(y)$ we will write $\cN_{F}$ for the number of nodal domains of $F$ which are contained in $\left(-\frac12,\frac12\right)^2$. However, during our argument we will need to consider nodal sets in a slightly bigger region, and for this reason we will always consider the norms of such functions on the bigger region $\left(-1,1\right)^2$.

We fix a (large) integer $K$ and divide the unit circle into arcs of length $\frac1{2K}$; specifically, identifying the unit circle with the interval $(-\frac12,\frac12]$ in the usual manner, we define the arcs
\begin{equation*}
  I_k=\bigg(\frac{k-1}{2K},\frac{k}{2K}\bigg]
\end{equation*}
for $-K+1\leq k\leq K$. (Note that for $1\leq k\leq K$ we have $\zeta\in I_k$ if and only if $-\zeta\in I_{K-k}$). Corresponding to these arcs we put
\begin{equation*}
  \cE^{(k)}=\{\xi\in\cE:\frac\xi {\sqrt E}\in I_k\}.
\end{equation*}

Fix $0<\delta<1$ and define
\begin{equation*}
  \K=\{-K+1\leq k\leq K: \mu_{E,a}(I_k)\geq\delta\}
\end{equation*}
and
\begin{equation*}
  \mathcal G=\cup_{k\notin\K}\cE^{(k)}.
\end{equation*}
We also write $\K^+$ for the set of positive $k\in\K$, and $\K^-$ for the non-positive terms (we have $k\in\K^+$ if and only if $k-K\in\K^-$). Note that for $k\in\K$ we have
\begin{equation*}
  \delta\leq\sum_{\xi\in\cE^{(k)}}|a_{\xi}|^2\leq M |\cE^{(k)}|
\end{equation*}
which implies that, under the assumptions that $N\to\infty$ as $E\to\infty$ and \eqref{eq: bound on coeff}, for fixed $\delta$,
\begin{equation}\label{eq: lots of freq in support}
  |\cE^{(k)}|\to\infty\text{ as }E\to\infty.
\end{equation}

For $k\in\K$ we define $\zeta^{(k)}$ to be the midpoint of $I_k$ (note that $\zeta^{(k)}=-\zeta^{(k+K)}$ for $k\in\K^-$). We split $f(x)=\widetilde{\psi}(x)+\tilde{f}(x)$ where
\begin{equation}\label{def: psitil}
  \widetilde{\psi}(x)=\sum_{\xi\in\mathcal G} a_\xi e(\langle \xi,x\rangle)
\end{equation}
and
\begin{equation*}
  \tilde{f}(x)= \sum_{k\in\K}\sum_{\xi\in\cE^{(k)}} a_\xi e(\langle \xi,x\rangle)
\end{equation*}
and correspondingly split (recall \eqref{eq:Fx def}) $F_x(y)=\widetilde{\psi}_{0,x}(y)+\tilde{F}_x(y)$ where
\begin{equation}\label{def: psi0til}
  \widetilde{\psi}_{0,x}(y)=\widetilde{\psi}\left(x+\frac R{\sqrt{E}}y\right)
\end{equation}
and
\begin{equation*}
  \tilde{F}_x(y)=\tilde{f}\left(x+\frac R{\sqrt{E}}y\right).
\end{equation*}
Note that
\begin{equation}\label{eq: Ftil simple expansion}
  \tilde{F}_x(y)= \sum_{k\in\K}\sum_{\xi\in\cE^{(k)}} a_\xi e\left(\left\langle \xi,x+\frac R{\sqrt{E}}y\right\rangle\right)
\end{equation}
and, since $\mu_{E,a}(I_k)\neq0$ for $k\in\K$, if we define
\begin{equation*}
  f_k(x,y)=\frac1{\mu_{E,a}(I_k)^{\frac12}} \sum_{\xi\in\cE^{(k)}} a_\xi e(\langle \xi,x\rangle) e\left(\left\langle R\left(\frac \xi{\sqrt{E}}-\zeta^{(k)}\right),y\right\rangle\right),
\end{equation*}
then
\begin{equation*}
  \tilde{F}_x(y)= \sum_{k\in\K} \mu_{E,a}(I_k)^{\frac12} f_k(x,y) e(\langle R\zeta^{(k)},y\rangle).
\end{equation*}
Define\footnote{Note the aforementioned difference between the definition of $b_k$ here and that of the previous section, we have normalised so that $\int_{\T^2}|b_k(x)|^2\,dx=1$. Henceforth, the definition given here is to be taken to be the definition of $b_k$}
\begin{equation}\label{eq: def bk}
  b_k(x) =f_k(x,0) =\frac1{\mu_{E,a}(I_k)^{\frac12}} \sum_{\xi\in\cE^{(k)}} a_\xi e(\langle \xi,x\rangle),
\end{equation}
and\footnote{Again note the presence of the factors $\mu_{E,a}(I_k)^{\frac12}$ and the restriction of the summation to the set $\K$ in the definition of $\phi_x$, which differs from the previous section.}
\begin{equation}\label{def: phi}
  \phi_x(y)= \sum_{k\in\K} \mu_{E,a}(I_k)^{\frac12} b_k(x) e(\langle R\zeta^{(k)},y\rangle);
\end{equation}
these coefficients satisfy $b_k=\bar{b}_{K-k}$ for $k\in\K^+$.

\subsection{Proof of Theorem \ref{thm: main thm mu E cgt}}

Our proof of Theorem \ref{thm: main thm mu E cgt} contains a number of different ingredients, which we now list here. First we state that most nodal domains of $f$ are `small' --- their diameter is at most of order $\frac 1{\sqrt E}$ --- which will allow us to study the number of nodal domains `locally', i.e., to study $\cN_f$ by counting the nodal domains of $F_x$. Recall that for a real-valued function $f$ on $\T^2$ we write $\cN_f(x,\epsilon)$ for the number of nodal domains of $f$ contained in the open box centred at $x$ of sidelength $\epsilon$.

\begin{lemma}[Most domains are small]\label{lem: localisation}
  For every $R>1$ and $E>R^2$ we have
  \begin{equation*}
  \cN_f=\frac E{R^2}\int_{\T^2}\cN_f\left(x,\frac R{\sqrt{E}}\right)\,dx+O\left(\frac ER\right).
  \end{equation*}
\end{lemma}

The proof of this lemma is given in Section~\ref{sec:Lemma local}. Recalling that $\cN_{F}$ denotes the number of nodal domains of $F$ which are contained in $\left(-\frac12,\frac12\right)^2$ we note that
$$\cN_f\left(x,\frac R{\sqrt{E}}\right)=\cN_{F_x}.$$
The next proposition, which will be proved in Section~\ref{sec:approx F by phi}, allows us to replace $F_x$ by (a perturbation of) $\phi_x$ (recall \eqref{def: phi}), for ``most'' $x$.

\begin{prop}[Approximating $F_{x}$ with $\phi_{x}$]\label{prop: gathering}
  Given $1<R<\sqrt{E}$ and $\epsilon_1,\epsilon_2>0$ there exist $K_0(R,\epsilon_1,\epsilon_2)$, $\delta_0(R,K,\epsilon_1,\epsilon_2)$ and $\psi_x\colon(-1,1)^2\to\R$ satisfying, for all $x\in\T^2$,
  \begin{rlist}
    \item $\|\psi_x\|_{\mathcal C^1}<\epsilon_1$ and
    \item $\Delta(\phi_x+\psi_x)=-4\pi^2R^2(\phi_x+\psi_x)$
  \end{rlist}
  such that for $K\geq K_0$ and $\delta\leq \delta_0$ we have
  \begin{equation*}
    \int_{\T^2}\cN_{F_x}\,dx=\int_{\T^2}\cN_{\phi_x+\psi_x}\,dx +O(\epsilon_2R^2).
  \end{equation*}
\end{prop}

Having reduced our problem to the study of the number of nodal domains of the field $\phi_x+\psi_x$, where
\begin{equation*}
  \phi_x(y)= \sum_{k\in\K} \mu_{E,a}(I_k)^{\frac12} b_k(x) e(\langle R\zeta^{(k)},y\rangle)
\end{equation*}
and $\psi_x$ is a small perturbation in the $\mathcal C^1$ norm, our next step is to apply a de-randomisation technique due to Bourgain \cite{B}. We will work on two different probability spaces, an abstract one $\Omega$, and $\T^2$ (equipped with the Lebesgue measure). We will refer to measurable maps from $\Omega$ as \emph{random}, and from $\T^2$ as \emph{quasi-random}. We will use $\Pro$ and $\E$ to refer to the abstract probability space $\Omega$, while $\Prr$ and $\Ee$ will refer to $\T^2$. ($\Prr$ is of course just Lebesgue measure, normalised appropriately.) We will always assume that $\omega\in\Omega$ and $x\in\T^2$. We view $\phi_x+\psi_x$ as a quasi-random field, and note that
\begin{equation*}
  \int_{\T^2}\cN_{\phi_x+\psi_x}\,dx=\Ee[\cN_{\phi+\psi}].
\end{equation*}

We show that it is enough to study the nodal domains of (small perturbations of) the random Gaussian field\footnote{Note again the difference between the definition here and that of the previous section.}
\begin{equation*}
  \Phi_\omega(y)= \sum_{k\in\K} \mu_{E,a}(I_k)^{\frac12} c_k(\omega) e(\langle R\zeta^{(k)},y\rangle),
\end{equation*}
where $c_k$ are i.i.d. standard (complex) Gaussian for $k\in\K^+$, and $c_k=\bar{c}_{k+K}$ for $k\in \K^-$. That is, the function $\Phi$ is constructed by replacing the quasi-random variables $b_k$ by the random variables $c_k$. The precise statement is the following, whose proof is given in Section~\ref{sec:derandom}:

\begin{prop}[Passage to nodal domains of random functions]\label{prop: pass to random}
Let $0<\epsilon_1,\epsilon_3,\delta<1$ and $K,R>1$ be given.  Let $\psi$ be any function satisfying (i) and (ii) of Proposition~\ref{prop: gathering}. Suppose that $N\to\infty$ as $E\to\infty$ and that the coefficients $a_\xi$ are of class $\cA(g)$, for some slowly growing function $g$. Then there exists $B_0=B_0(R,K,\epsilon_1,\epsilon_3)$ and $E_1=E_1(R,K,\epsilon_1,\epsilon_3,\delta,\gamma,g)$ such that if $\cE$ satisfies $I(\gamma,B_0)$ for some fixed $\gamma$,
there exists a random $\Psi_\omega\colon[-1,1]^2 \to \R$ such that for all $E\geq E_1$:
\begin{rlist}
\item $\|\Psi_\omega\|_{\mathcal C^1}^2\leq 2\epsilon_1$ for all $\omega$.

\item $\Delta(\Phi_\omega+\Psi_\omega)=-4\pi^2R^2(\Phi_\omega+\Psi_\omega)$  for all $\omega$ (so $\Psi_\omega$ is in $\mathcal C^2$ for
all $\omega$).

\item There exists $\Omega'\subset\Omega$ and a measure-preserving map $\tau:\Omega'\to\T^2$ satisfying

\begin{itemize}
\item $\Pro[\Omega\setminus\Omega']=O(\epsilon_3)$,

\item $\Prr[\T^2\setminus\tau(\Omega')]=O(\epsilon_3)$ and

\item $\Phi_\omega+\Psi_\omega=\phi_{\tau(\omega)}+\psi_{\tau(\omega)}$ for all $\omega\in\Omega'$.

\end{itemize}

\item $\E[\cN_{\Phi+\Psi}]=\Ee[\cN_{\phi+\psi}]+O(\epsilon_3 R^2)$.

\end{rlist}

\end{prop}

Our final ingredient is an application of the work of Nazarov and Sodin \cite{NaSo2}, who have comprehensively studied nodal domains of Gaussian fields. We do this in two steps, the details of which are given in Section~\ref{sec: Naz Sod}; firstly we apply their techniques to see that the perturbation $\Psi$ does not play any significant r\^{o}le.

\begin{prop}[$\Psi$ does not contribute]\label{prop: Psi negligible}
  Let $\epsilon_4>0$ and $\Phi$ be given, and suppose that $\Psi$ is a random function satisfying Proposition~\ref{prop: pass to random} (i) and (ii) for a sufficiently small $\epsilon_1$ (depending on $R,\epsilon_4$ and $\mu$). Then there exist $K_1(\mu),\delta_1(K,\mu)$ and $E_2(K,\delta)$ such that
  \begin{equation*}
    \E[\cN_{\Phi+\Psi}]=\E[\cN_{\Phi}]+O(\epsilon_4 R^2 + R)
  \end{equation*}
  for $K>K_1,\delta<\delta_1$ and $E>E_2$.
\end{prop}

Finally we apply Nazarov and Sodin's results to compute (asymptotically) the expected number of nodal domains of $\Phi$.

\begin{prop}\label{prop: NS constant}
  Given $\epsilon_5>0$ there exists $K_2(\epsilon_5)$ and $E_3(\epsilon_5)$ such that for all $K\geq K_2$ and $E\geq E_3$
  \begin{equation*}
    \E[\cN_\Phi]= c_{NS}(\mu)R^2 +O(\epsilon_5 R^2 + R)
  \end{equation*}
  provided that $\delta\leq K^{-2}$.
\end{prop}

\begin{proof}[Proof of Theorem \ref{thm: main thm mu E cgt} assuming Lemma \ref{lem: localisation}
and Propositions \ref{prop: gathering}-\ref{prop: NS constant}]

Let $\epsilon>0$. Choose $R=\frac1\epsilon$ and $\epsilon_5=\epsilon$. Applying Proposition~\ref{prop: NS constant} we see that
\begin{equation*}
    \E[\cN_\Phi]= c_{NS}(\mu)R^2 +O(R)
\end{equation*}
for $K\geq K_2(\epsilon),E\geq E_3(\epsilon)$ and  $\delta\leq K^{-2}$. Choosing $\epsilon_4=\epsilon$ we have, by Proposition~\ref{prop: Psi negligible},
\begin{equation}\label{eq: put toget1}
    \E[\cN_{\Phi+\Psi}]= c_{NS}(\mu)R^2 +O(R)
\end{equation}
for $K\geq \max\{K_1(\mu),K_2\}, E>\max\{E_2(K,\delta)),E_3\},\delta<\min\{K^{-2},\delta_1(K,\mu)\}$, and for any random $\Psi$ satisfying Proposition~\ref{prop: pass to random} (i) and (ii) for a sufficiently small $\epsilon_1$ (depending on $\epsilon$ and $\mu$).

We fix this value of $\epsilon_1$ and choose $\epsilon_2=\epsilon$. From Proposition~\ref{prop: gathering} this choice yields $K_0(\epsilon,\mu)$ and we now fix $K=\max\{K_0,K_1,K_2\}$. Furthermore, with this fixed value of $K$, Proposition~\ref{prop: gathering} also yields $\delta_0(\epsilon,\mu)$ and we also fix $\delta<\min\{K^{-2},\delta_0,\delta_1\}$. Applying Proposition~\ref{prop: gathering} we have
\begin{equation}\label{eq: put toget2}
  \Ee[\cN_{F}]=\Ee[\cN_{\phi+\psi}] + O(R)
\end{equation}
for some $\psi$ satisfying Proposition~\ref{prop: gathering} (i) and (ii) and for all $E\geq \sqrt R$. We choose $\epsilon_3=\epsilon$ and apply Proposition~\ref{prop: pass to random} to see that there exist $B_0(\epsilon,\mu)$, $E_1(\epsilon,\gamma,g,\mu)$ and a random $\Psi$ satisfying Proposition~\ref{prop: pass to random} (i) and (ii) such that
\begin{equation}\label{eq: put toget3}
  \Ee[\cN_{\phi+\psi}]= \E[\cN_{\Phi+\Psi}] + O(R)
\end{equation}
provided that $E\geq E_1$ and $\cE$ satisfies $I(\gamma,B_0)$. Combining \eqref{eq: put toget1}, \eqref{eq: put toget2} and \eqref{eq: put toget3} we have
\begin{equation*}
  \Ee[\cN_{F}]=c_{NS}(\mu)R^2 + O(R)
\end{equation*}
for all $E\geq \max\{E_1,E_2,E_3,\sqrt R\}$, provided that $\cE$ satisfies $I(\gamma,B_0)$.

Since $B(E)\to\infty$ we may find $E_4(\epsilon)$ such that $B(E)\geq B_0$ for all $E\geq E_4$. By hypothesis, $\cE$ satisfies $I(\gamma,B_0)$ for all such $E$. Choosing $E_0 = \max\{E_1,E_2,E_3,E_4,\sqrt R\}$ we have
\begin{equation*}
  \Ee[\cN_{F}]=c_{NS}(\mu)R^2 + O(R)
\end{equation*}
for all $E\geq E_0$. Finally we apply Lemma~\ref{lem: localisation} to get
\begin{equation*}
  \cN_{f}=c_{NS}(\mu)E + O\left(\frac ER\right)= c_{NS}(\mu)E + O(\epsilon E)
\end{equation*}
by our choice of $R$. This completes the proof.
\end{proof}

\section{Proof of Lemma \ref{lem: localisation}: most components arise locally}\label{sec:Lemma local}

\begin{proof}
The proof of this lemma is essentially the same as that of the so-called `integral-geometric sandwich' (see \cite{NaSo2}*{Lemma 1}. Since we are using boxes rather than balls, we give a direct proof for the reader's convenience.
The results of Donnelly and Fefferman \cite{DF}*{Theorem 1.2} show that the length of
\begin{equation*}
\mathcal Z(f)=\{x\in\T^2:f(x)=0\}
\end{equation*}
is $O(\sqrt E)$. In particular, the number of nodal domains of diameter at least $\frac1{\epsilon\sqrt{E}}$ is $O(\epsilon E)$, for any $\epsilon>0$.

Now if $\pi_1$ (respectively $\pi_2$) is the projection from $\T^2$ onto the first (respectively second) coordinate and $\epsilon>0$, then we write $\mathcal D_\epsilon$ for the set of nodal domains, $D$, of $f$ satisfying $|\pi_1(D)|,|\pi_2(D)|<\epsilon$, where $|\pi_j(D)|$ refers to the length of the interval $\pi_j(D)$ as a subset of $\T$. We define, for $x\in\T^2$ and $D\in\mathcal D_\epsilon$
  \begin{equation*}
    \chi_D(x,\epsilon)=
    \begin{cases}
      1,&\textrm{ if }D\textrm{ is contained in the open box centred at }x\textrm{ of sidelength }\epsilon\\
      0,&\textrm{ otherwise}
    \end{cases}.
  \end{equation*}
  We clearly have
  \begin{equation*}
    \cN_f\left(x,\epsilon\right)=\sum_{D\in\mathcal D_\epsilon}\chi_D(x,\epsilon)
  \end{equation*}
  and, since
  \begin{equation*}
    \int_{\T^2}\chi_D(x,\epsilon)\,dx=(\epsilon-|\pi_1(D)|)(\epsilon-|\pi_2(D)|),
  \end{equation*}
  we see that
  \begin{equation*}
    \int_{\T^2}\cN_f\left(x,\epsilon\right)\,dx=\sum_{D\in\mathcal D_\epsilon}(\epsilon-|\pi_1(D)|)(\epsilon-|\pi_2(D)|).
  \end{equation*}
  Taking $\epsilon=\frac R{\sqrt{E}}$ we have
  \begin{equation*}
    \int_{\T^2}\cN_f\left(x,\frac R{\sqrt{E}}\right)\,dx = \sum \left(\frac R{\sqrt{E}}-|\pi_1(D)|\right) \left(\frac R{\sqrt{E}}-|\pi_2(D)|\right)
  \end{equation*}
  where the sum runs over all nodal domains, $D$, of $f$ satisfying $|\pi_1(D)|,|\pi_2(D)|<\frac R{\sqrt{E}}$. Letting $\cN_{\text{small}}$ be the number of such domains we have
  \begin{equation*}
    \int_{\T^2}\cN_f\left(x,\frac R{\sqrt{E}}\right)\,dx = \frac {R^2}E \cN_{\text{small}} - \frac R{\sqrt{E}} \sum\left( |\pi_1(D)|+|\pi_2(D)|\right) + \sum |\pi_1(D)||\pi_2(D)|.
  \end{equation*}
  Now
  \begin{equation*}
    \sum |\pi_1(D)||\pi_2(D)| \leq \frac R{\sqrt{E}} \sum |\pi_1(D)|
  \end{equation*}
  and
  \begin{equation*}
    \sum |\pi_j(D)| \lesssim |\mathcal Z(f)| \lesssim \sqrt E
  \end{equation*}
  for $j=1,2$.

  We write $\cN_{\text{big}}$ for the number of nodal domains with either $|\pi_1(D)| \geq \frac R{\sqrt{E}}$ or $|\pi_2(D)|\geq\frac R{\sqrt{E}}$, and note that $\cN_{\text{big}}\lesssim\frac ER$ since such domains have diameter at least $\frac R{\sqrt{E}}$. Thus
  \begin{equation*}
    \cN_{f} = \cN_{\text{small}}+\cN_{\text{big}} = \cN_{\text{small}} + O\left(\frac ER\right) = \frac E{R^2} \int_{\T^2}\cN_f\left(x,\frac R{\sqrt{E}}\right)\,dx + O\left(\frac ER\right).
  \end{equation*}
\end{proof}

\section{Proof of Proposition \ref{prop: gathering}: approximating $F_{x}$ with $\phi_{x}$}\label{sec:approx F by phi}

We show that, for ``most'' $x$, we can think of $F_x$ as a ``small'' perturbation of $\phi_x$ in the $\mathcal C^1$-norm. This is the content of the next two lemmas.

Here, and throughout, the notation $\|\widetilde{\psi}_{0,x}\|$ for a norm $\|\cdot\|$ means that we are fixing an $x$ and considering the norm of $\widetilde{\psi}_{0,x}$ as a function of $y$. Similar remarks apply to $\tilde{F}_x$, $\phi_x$, etc. Furthermore, we recall that we consider the norm over $y\in\left(-1,1\right)^2$, for example,
\begin{equation*}
  \|\widetilde{\psi}_{0,x}\|_{L^2} = \left(\int_{(-1,1)^2} \Abs{\widetilde{\psi}_{0,x}(y)}^2 \,dy \right)^{1/2}.
\end{equation*}

\begin{lemma}[Frequencies outside the support of the limiting measure don't contribute]\label{lem: outside support negligible}
  Let $R,K$ and $\delta$ be as before. Then
  \begin{equation*}
    \int_{\T^2}\|\widetilde{\psi}_{0,x}\|_{\mathcal C^1}^2\,dx =O(R^{6}\delta K).
  \end{equation*}
\end{lemma}

\begin{lemma}[We may replace $f_k$ by $b_k$; cf. \cite{B}*{Lemma 1}]\label{lem: F to phi}
  For fixed $R\geq1$ and $K$,
  \begin{equation*}
    \int_{\T^2}\|\tilde{F}_x-\phi_x\|_{\mathcal C^1}^2\,dx=O\left(\frac{R^{8}}{K^2}\right).
  \end{equation*}
\end{lemma}

We first show how Lemmas~\ref{lem: outside support negligible} and \ref{lem: F to phi} imply Proposition~\ref{prop: gathering}, we will then prove the lemmas.

\begin{proof}[Proof of Proposition \ref{prop: gathering}, assuming Lemmas \ref{lem: outside support negligible} and \ref{lem: F to phi}]
  By Lemma~\ref{lem: F to phi} we have
  \begin{equation*}
    \|\tilde{F}_x-\phi_x\|_{\mathcal C^1}<\frac{\epsilon_1}2
  \end{equation*}
  for all $x$ outside a set $V_1$ of measure at most $CR^8/(K^2\epsilon_1^2)<\epsilon_2$, for appropriately large $K$. Similarly, by Lemma~\ref{lem: outside support negligible} we have
  \begin{equation*}
    \|\widetilde{\psi}_{0,x}\|_{\mathcal C^1}<\frac{\epsilon_1}2
  \end{equation*}
  for all $x$ outside a set $V_2$ of measure at most
  \begin{equation*}
    \frac{CR^6\delta K}{\epsilon_1^2}<\epsilon_2,
  \end{equation*}
  for appropriately small $\delta$.

  Write $V=V_1\cup V_2$ and define
  \begin{equation*}
    \psi_x=
    \begin{cases}
      F_x-\phi_x&\text{for }x\in\T^2\setminus V\\
      0&\text{for }x\in V
    \end{cases}.
  \end{equation*}
  Since $F_x=\tilde{F}_x+\widetilde{\psi}_{0,x}$ we manifestly have $\|\psi_x\|_{\mathcal C^1}<\epsilon_1$ for all $x$. Furthermore, since $\Delta F_x=-4\pi^2 R^2 F_x$ and $\Delta \phi_x=-4\pi^2 R^2 \phi_x$ for all $x$, we have
  \begin{equation*}
    \Delta(\phi_x+\psi_x)=-4\pi^2R^2(\phi_x+\psi_x)
  \end{equation*}
  for all $x$.

  Trivially
  \begin{align*}
    \int_{\T^2}\cN_{F_x}\,dx&=\int_{\T^2\setminus V}\cN_{F_x}\,dx+\int_{V}\cN_{F_x}\,dx\\
    &=\int_{\T^2\setminus V}\cN_{\phi_x+\psi_x}\,dx+\int_{V}\cN_{F_x}\,dx\\
    &=\int_{\T^2}\cN_{\phi_x+\psi_x}\,dx-\int_{V}\cN_{\phi_x+\psi_x}\,dx+\int_{V}\cN_{F_x}\,dx.
  \end{align*}
  By Courant's Nodal Theorem $\cN_{F_x}=O(R^2)$ and $\cN_{\phi_x+\psi_x}=O(R^2)$ which yields
  \begin{equation*}
    \int_{\T^2}\cN_{F_x}\,dx=\int_{\T^2}\cN_{\phi_x+\psi_x}\,dx+O(\mes(V) R^2)=\int_{\T^2}\cN_{\phi_x+\psi_x}\,dx+O(\epsilon_2 R^2),
  \end{equation*}
  as claimed.
\end{proof}

\begin{proof}[Proof of Lemma~\ref{lem: outside support negligible}]
  By standard Sobolev estimates we have
  \begin{equation*}
    \|\widetilde{\psi}_{0,x}\|_{\mathcal C^1}^2\leq C\|\widetilde{\psi}_{0,x}\|_{H^{3}}^2=C\sum_{|\alpha|\leq 3}\|D^\alpha\widetilde{\psi}_{0,x}\|_{L^{2}}^2,
  \end{equation*}
  where the sum runs over multi-indices $\alpha=(\alpha_1,\alpha_2)$ and $D^\alpha=\frac{\partial^{|\alpha|}}{\partial y_1^{\alpha_1}\partial y_2^{\alpha_2}}$. From \eqref{def: psitil} and \eqref{def: psi0til} it is clear that
  \begin{equation*}
    D^\alpha\widetilde{\psi}_{0,x}(y)=\sum_{\xi\in\mathcal G} a_\xi e(\langle \xi,x\rangle) D^\alpha e \left( \left\langle \xi,\frac R{\sqrt{E}}y \right\rangle \right).
  \end{equation*}
  Combining these and applying Fubini yields
  \begin{align}\label{eq:long expr lem osn}
    \int_{\T^2}\|\widetilde{\psi}_{0,x}\|_{\mathcal C^3}^2\,dx\leq C\sum_{|\alpha|\leq 3} \sum_{\xi,\xi'\in\mathcal G} a_\xi \bar{a}_{\xi'}& \int_{\T^2} e(\langle \xi-\xi',x\rangle) \,dx\notag\\
    &\quad  \int_{[-1,1]^2} D^\alpha e\left(\left\langle \xi,\frac R{\sqrt{E}}y\right\rangle\right) \overline{D^\alpha e\left(\left\langle \xi',\frac R{\sqrt{E}}y\right\rangle\right)}\,dy.
  \end{align}
  We compute that
  \begin{equation*}
    \int_{\T^2} e(\langle \xi-\xi',x\rangle) \,dx=\delta_{\xi\xi'},
  \end{equation*}
  where $\delta_{\xi\xi'}$ is the Kronecker delta, and
  \begin{equation*}
    D^\alpha e\left(\left\langle \xi,\frac R{\sqrt{E}}y\right\rangle\right)= \left(2\pi i \frac R{\sqrt{E}}\right)^{|\alpha|}\xi_1^{\alpha_1}\xi_2^{\alpha_2} e\left(\left\langle \xi,\frac R{\sqrt{E}}y\right\rangle\right).
  \end{equation*}
  Substituting these into \eqref{eq:long expr lem osn} yields
  \begin{equation*}
    \int_{\T^2}\|\widetilde{\psi}_{0,x}\|_{\mathcal C^1}^2\,dx\leq C\sum_{|\alpha|\leq 3} \sum_{\xi\in\mathcal G} |a_\xi|^2 \left(2\pi \frac R{\sqrt{E}}\right)^{2|\alpha|}\xi_1^{2\alpha_1}\xi_2^{2\alpha_2}\leq C R^{6}\sum_{\xi\in\mathcal G} |a_\xi|^2,
  \end{equation*}
  and it is clear that
  \begin{equation*}
    \sum_{\xi\in\mathcal G} |a_\xi|^2=\mu_{E,a}(\cup_{k\notin\K}I_k)=\sum_{k\notin\K}\mu_{E,a}(I_k).
  \end{equation*}

  By definition, $\mu_{E,a}(I_k)\leq\delta$ for $k\notin\K$, and there cannot be more than $2K$ summands $k\notin\K$. This completes the proof.
\end{proof}

\begin{proof}[Proof of Lemma~\ref{lem: F to phi}]
  As before we estimate the $\mathcal C^1$-norm by the $H^{3}$-norm. We have
  \begin{equation*}
    \int_{\T^2}\|\tilde{F}_x-\phi_x\|_{\mathcal C^1}^2\,dx\leq C\sum_{|\alpha|\leq 3} \int_{\T^2}\|D^\alpha(\tilde{F}_x-\phi_x)\|_{L^{2}}^2\,dx
  \end{equation*}
  and clearly by \eqref{eq: Ftil simple expansion}
  \begin{equation*}
    D^\alpha\tilde{F}_x= \sum_{k\in\K}\sum_{\xi\in\cE^{(k)}} a_\xi e(\langle\xi,x\rangle) D^\alpha e\left(\left\langle\xi,\frac R{\sqrt{E}}y\right\rangle\right)
  \end{equation*}
  and by \eqref{eq: def bk} and \eqref{def: phi}
  \begin{equation*}
    D^\alpha\phi_x= \sum_{k\in\K}\sum_{\xi\in\cE^{(k)}} a_\xi e(\langle\xi,x\rangle) D^\alpha e\left(\left\langle R\zeta^{(k)},y\right\rangle\right).
  \end{equation*}
  To simplify the notation we write $D^\alpha(\xi,\zeta^{(k)})= D^\alpha \left( e\left(\left\langle R\frac\xi{\sqrt{E}},y\right\rangle\right) - e\left(\left\langle R\zeta^{(k)},y\right\rangle\right)\right)$ and then have
  \begin{equation*}
    |D^\alpha(\tilde{F}_x-\phi_x)|^2= \sum_{k,k'\in\K}\sum_{\xi\in\cE^{(k)}}\sum_{\xi'\in\cE^{(k')}} a_\xi\bar{a}_{\xi'} e(\langle\xi-\xi',x\rangle) D^\alpha (\xi,\zeta^{(k)}) \overline{D^\alpha (\xi',\zeta^{(k')})}
  \end{equation*}
  which yields
  \begin{align*}
    \int_{\T^2}\int_{[-1,1]^2} &|D^\alpha(\tilde{F}_x-\phi_x)|^2\,dy\,dx\\
    &= \sum_{k,k'\in\K}\sum_{\xi\in\cE^{(k)}}\sum_{\xi'\in\cE^{(k')}} a_\xi\bar{a}_{\xi'} \int_{\T^2}e(\langle\xi-\xi',x\rangle) \,dx \int_{[-1,1]^2} D^\alpha (\xi,\zeta^{(k)}) \overline{D^\alpha (\xi',\zeta^{(k')})}\,dy.
  \end{align*}
  As before we compute
  \begin{equation*}
    \int_{\T^2} e(\langle \xi-\xi',x\rangle) \,dx=\delta_{\xi\xi'}
  \end{equation*}
  so that
  \begin{equation*}
    \int_{\T^2}\int_{[-1,1]^2} |D^\alpha(\tilde{F}_x-\phi_x)|^2\,dy\,dx= \sum_{k\in\K}\sum_{\xi\in\cE^{(k)}} |a_\xi|^2\int_{[-1,1]^2} |D^\alpha (\xi,\zeta^{(k)})|^2\,dy.
  \end{equation*}
  Moreover
  \begin{align*}
    D^\alpha(\xi,\zeta^{(k)})= (&2\pi i R)^{|\alpha|}\\
    & \left( \left(\frac{\xi_1}{\sqrt{E}}\right)^{\alpha_1} \left(\frac{\xi_2}{\sqrt{E}}\right)^{\alpha_2} e\left(\left\langle R\frac\xi{\sqrt{E}},y\right\rangle\right) - \left(\zeta_1^{(k)}\right)^{\alpha_1} \left(\zeta_2^{(k)}\right)^{\alpha_2} e\left(\left\langle R\zeta^{(k)},y\right\rangle\right) \right)
  \end{align*}
  from which we deduce, recalling that $|\zeta^{(k)}|=1$, that
  \begin{align*}
    |D^\alpha(\xi,&\zeta^{(k)})|=\\
    &(2\pi R)^{|\alpha|} \Abs{ \left(\frac{\xi_1}{\sqrt{E}}\right)^{\alpha_1} \left(\frac{\xi_2}{\sqrt{E}}\right)^{\alpha_2} - \left(\zeta_1^{(k)}\right)^{\alpha_1} \left(\zeta_2^{(k)}\right)^{\alpha_2} e\left(\left\langle R\left(\zeta^{(k)}-\frac\xi{\sqrt{E}}\right),y\right\rangle\right) }\\
    &\leq (2\pi R)^{|\alpha|} \Bigg( \Abs{\left(\frac{\xi_1}{\sqrt{E}}\right)^{\alpha_1} \left(\frac{\xi_2}{\sqrt{E}}\right)^{\alpha_2} - \left(\zeta_1^{(k)}\right)^{\alpha_1} \left(\zeta_2^{(k)}\right)^{\alpha_2}}\\
    &\qquad \qquad\qquad\qquad\qquad\qquad  +\Abs{\zeta_1^{(k)}}^{\alpha_1}\Abs{\zeta_2^{(k)}}^{\alpha_2}\Abs{ 1- e\left(\left\langle R\left(\zeta^{(k)}-\frac\xi{\sqrt{E}}\right),y\right\rangle\right) } \Bigg)\\
    &\leq CR^{|\alpha|}\left(\Abs{\zeta^{(k)}-\frac\xi{\sqrt{E}}}+R\Abs{\zeta^{(k)}-\frac\xi{\sqrt{E}}}\right) \leq C\frac{R^{|\alpha|+1}}{K}.
  \end{align*}
  We therefore have
  \begin{align}\label{eq: estimate to sum}
    \int_{\T^2}\int_{(-\frac12,\frac12)^2} |D^\alpha(\tilde{F}_x-\phi_x)|^2\,dy\,dx &\leq C\frac{R^{2|\alpha|+2}}{K^2}\sum_{k\in\K}\sum_{\xi\in\cE^{(k)}} |a_\xi|^2\notag\\
    &=C\frac{R^{2|\alpha|+2}}{K^2}\mu_{E,a}(\cup_{k\in\K}I_k) \leq C\frac{R^{2|\alpha|+2}}{K^2}.
  \end{align}
  Summing \eqref{eq: estimate to sum} over the indices $\alpha$ with $|\alpha|\leq 3$ yields the claimed result.
\end{proof}

\section{Proof of Proposition \ref{prop: pass to random}: Bourgain's de-randomisation}\label{sec:derandom}

\subsection{Bourgain's de-randomisation technique I, approximately Gaussian coefficients}

Our next step is to relate the number of nodal of domains of $\phi_x+\psi_x$ to the nodal domains of a random function. We will do this by seeing that we may treat the quantities $b_k(x)$ (recall \eqref{eq: def bk}) as being ``approximately Gaussian'' when we think of $\T^2$ as a probability space. Recall that we work on two different probability spaces, an abstract one $\Omega$, and $\T^2$, that we refer to measurable maps from $\Omega$ as random, and from $\T^2$ as quasi-random, that $\Pro$ and $\E$ refer to the abstract probability space $\Omega$, while $\Prr$ and $\Ee$ refer to $\T^2$ ($\Prr$ is of course just Lebesgue measure, normalised appropriately) and that in this notation
\begin{equation*}
  \int_{\T^2}\cN_{\phi_x+\psi_x}\,dx=\Ee[\cN_{\phi+\psi}].
\end{equation*}
We will always assume that $\omega\in\Omega$ and $x\in\T^2$.

In this subsection we show the following proposition, which states that, given a sequence $\{c_k\}_{k\in \K^+}$ of iid $N_\C(0,1)$ random variables (that is, the probability density of each $c_k$ is $\frac{1}{\pi}e^{-|z|^2}$ with respect to the Lebesgue measure on the plane), a bounded number of moments of the sequence of $\C$-valued quasi-random variables $\{b_k\}_{k\in \K^+}$ (recall \eqref{eq: def bk}) are ``close'' to the moments of the sequence $\{c_k\}_{k\in \K^+}$.

\begin{prop}[The moments of $b_k$ are approximately Gaussian]\label{prop:eps  Gauss mom}
  Let $\delta,B_1$ and $\epsilon_6>0$ be given. Suppose that $N\to\infty$ as $E\to\infty$, that $\cE$ satisfies $I(\gamma,B_1)$ for all $E$ and for some $\gamma<\frac12$, and that the coefficients $a_\xi$ are of class $\cA(g)$, for some slowly growing function $g$. There exists $E_5=E_5(\epsilon_6,\delta,B_1,\gamma,g)$ such that if $r_k\geq0$ and $s_k\geq0$ satisfy $\sum_{k\in\K^+} r_k+s_k < B_1$ then for all $E\geq E_5$
  \begin{equation*}
    \left|\E\left[\prod_{k\in\K^+} c_k^{r_k}\bar{c}_k^{s_k}\right]-\Ee\left[\prod_{k\in\K^+} b_k^{r_k}\bar{b}_k^{s_k}\right]\right|\leq \epsilon_6.
  \end{equation*}
\end{prop}

\begin{proof}
  A routine calculation shows that
  \begin{equation}\label{eq: Gauss mom}
    \E\left[\prod_{k\in\K^+} c_k^{r_k}\bar{c}_k^{s_k}\right]=\prod_{k\in\K^+} \E\left[ c_k^{r_k}\bar{c}_k^{s_k}\right]=\prod_{k\in\K^+} \delta_{r_k,s_k}r_k!,
  \end{equation}
  we compute here the moments of $b_k$. We begin by noting that
  \begin{align}\label{eq: bk mom}
    \Ee\left[\prod_{k\in\K^+} b_k^{r_k}\bar{b}_k^{s_k}\right]=\prod_{k\in\K^+}& \mu_{E,a}(I_k)^{-\frac{r_k+s_k}{2}}\notag \\
    &\Ee\left[\prod_{k\in\K^+} \bigg(\sum_{\xi\in\cE^{(k)}} a_\xi e(\langle \xi,x\rangle) \bigg)^{r_k} \bigg(\sum_{\xi\in\cE^{(k)}} \bar{a}_\xi e(\langle -\xi,x\rangle) \bigg)^{s_k} \right]
  \end{align}
  and we compute that
  \begin{align*}
    \Ee&\Bigg[\prod_{k\in\K^+} \bigg(\sum_{\xi\in\cE^{(k)}} a_\xi e(\langle \xi,x\rangle) \bigg)^{r_k} \bigg(\sum_{\xi\in\cE^{(k)}} \bar{a}_\xi e(\langle -\xi,x\rangle) \bigg)^{s_k} \Bigg]=\\
    &\sum \left( \prod_{k\in\K^+} \prod_{n=1}^{r_k}a_{\xi_{k,n}}\prod_{m=1}^{s_k}\bar{a}_{\xi'_{k,m}} \right)\Ee\left[e\bigg(\bigg\langle \sum_{k\in\K^+}(\xi_{k,1}+\cdots+\xi_{k,r_k}-\xi'_{k,1}-\cdots-\xi'_{k,s_k}),x\bigg\rangle\bigg)\right]
  \end{align*}
  where the unlabelled summation runs over all choices of $\xi_{k,1},\dots,\xi_{k,r_k},\xi'_{k,1},\dots,\xi'_{k,s_k}\in \cE^{(k)}$ for every $k\in\K^+$. Now, for $\xi\in\Z^2$, we have
  \begin{equation*}
    \Ee[e(\langle \xi,x\rangle)]=\delta_{\xi,0}
  \end{equation*}
  and so
  \begin{align}\label{eq: moment vanishing sum}
    \Ee\Bigg[\prod_{k\in\K^+} \left(\sum_{\xi\in\cE^{(k)}} a_\xi e(\langle \xi,x\rangle)\right)^{r_k} & \left(\sum_{\xi\in\cE^{(k)}} \bar{a}_\xi e(\langle -\xi,x\rangle)\right)^{s_k} \Bigg]\notag \\
    &=\sum \prod_{k\in\K^+}\left(\prod_{n=1}^{r_k}a_{\xi_{k,n}}\prod_{m=1}^{s_k}\bar{a}_{\xi'_{k,m}}\right),
  \end{align}
  where the sum runs only over the $\xi_{k,1},\dots,\xi_{k,r_k},\xi'_{k,1},\dots,\xi'_{k,s_k}\in\cE^{(k)}$ that satisfy
  \begin{equation*}
    \sum_{k\in\K^+} (\xi_{k,1}+\cdots+\xi_{k,r_k}-\xi'_{k,1}-\cdots-\xi'_{k,s_k}) =0.
  \end{equation*}

  We split the sum into two separate cases, the ``diagonal'' contribution which correspond to the case when the tuple
  \begin{equation*}
    \xi'_{k,1},\dots,\xi'_{k,s_k}
  \end{equation*}
  is a re-ordering of the tuple
  \begin{equation*}
    \xi_{k,1},\dots,\xi_{k,r_k},
  \end{equation*}
  and the remaining ``off-diagonal'' contribution.

  \begin{claim}\label{cl: first claim}
    There exists $E_6=E_6(\delta,B_1,g)$ such that
    \begin{align*}
      \left(\prod_{k\in\K^+} \mu_{E,a}(I_k)^{-r_k-s_k}\right) \sum_{\mathrm{diagonal}} \prod_{k\in\K^+}& \left(\prod_{n=1}^{r_k}a_{\xi_{k,n}}\prod_{m=1}^{s_k}\bar{a}_{\xi'_{k,m}}\right)
       = \left(\prod_{k\in\K^+} \delta_{r_k,s_k}r_k!\right) + O_{B_1}\left(\frac{M}{\delta}\right)
    \end{align*}
    for all $E\geq E_6$.
  \end{claim}

  \begin{claim}\label{cl: second claim}
    \begin{equation*}
      \prod_{k\in\K^+} \mu_{E,a}(I_k)^{-\frac{r_k+s_k}{2}} \Abs{\sum_{\mathrm{off-diagonal}} \prod_{k\in\K^+}\left(\prod_{n=1}^{r_k}a_{\xi_{k,n}}\prod_{m=1}^{s_k}\bar{a}_{\xi'_{k,m}}\right)}\leq \frac{C(B_1)}{\delta^{B_1/2}} \frac{g(N)^{B_1/2} } {N^{3(\frac12-\gamma)}}.
    \end{equation*}
  \end{claim}

  Combining these two claims with \eqref{eq: Gauss mom}, \eqref{eq: bk mom} and \eqref{eq: vanishing relation} we see that
  \begin{equation*}
    \left|\E\left[\prod_{k\in\K^+} c_k^{r_k}\bar{c}_k^{s_k}\right]-\Ee\left[\prod_{k\in\K^+} b_k^{r_k}\bar{b}_k^{s_k}\right]\right|\leq C(B_1) \left( \frac M\delta + \frac{1}{\delta^{B_1/2}} \frac{g(N)^{B_1/2} } {N^{3(\frac12-\gamma)}}\right).
  \end{equation*}
  By hypothesis \eqref{eq: bound on coeff} and our assumption that $N\to \infty$ we have $M\to0$ and $\frac{g(N)^{B_1/2} } {N^{3(\frac12-\gamma)}}\to 0$ (since $\gamma<\frac12$) as $E\to\infty$. This proves Proposition~\ref{prop:eps  Gauss mom}.
\end{proof}

It remains to prove Claims~\ref{cl: first claim} and \ref{cl: second claim}.
\begin{proof}[Proof of Claim~\ref{cl: first claim}]
  Note that in the diagonal case $r_k=s_k$ and $\prod_{n=1}^{r_k}a_{\xi_{k,n}}=\prod_{m=1}^{s_k}a_{\xi'_{k,m}}$, in which case
  \begin{equation*}
    \sum_{\text{diagonal}} \prod_{k\in\K^+}\left(\prod_{n=1}^{r_k}a_{\xi_{k,n}}\prod_{m=1}^{s_k}\bar{a}_{\xi'_{k,m}}\right) = \prod_{k\in\K^+} \sideset{}{^{(k)}}\sum \prod_{n=1}^{r_k}|a_{\xi_{k,n}}|^2
  \end{equation*}
  where $\sideset{}{^{(k)}}\sum$ denotes a sum over all tuples $\xi_{k,1},\dots,\xi_{k,r_k}\in\cE^{(k)}$ and re-orderings of that tuple $\xi'_{k,1},\dots,\xi'_{k,r_k}$. We enumerate the elements of $\cE^{(k)}$ as $\Xi_{k,1},\dots,\Xi_{k,N_k}$ and assume that $E$ is large enough that $N_k\geq r_k$ for all $k\in\K^+$ (recall \eqref{eq: lots of freq in support}). Corresponding to every tuple there is an integer $1\leq u_k \leq r_k$ such that exactly $u_k$ distinct elements of $\cE^{(k)}$ appear in the tuple. We let $1\leq n_{k,1}<\dots<n_{k,u_k}\leq N_k$ be the integers such that $\Xi_{k,n_{k,\alpha}}$ appears in the tuple and denote by $m_{k,\alpha}$ the number of times it appears. Note that $m_{k,1}+\dots+m_{k,u_k}=r_k$, that
  \begin{equation*}
    \prod_{n=1}^{r_k}|a_{\xi_{k,n}}|^2 = \prod_{\alpha=1}^{u_k}|a_{\Xi_{k,n_{k,\alpha} } } |^{2m_{k,\alpha}},
  \end{equation*}
  and that, given the values $u_k,n_{k,1},\dots,n_{k,u_k},m_{k,1},\dots,m_{k,u_k}$, there are
  \begin{equation*}
    \begin{pmatrix}
      r_k \\
      m_{k,1}\dots m_{k,u_k}
    \end{pmatrix}
    =\frac{r_k !}{\prod_{\alpha=1}^{u}m_{k,\alpha} !}
  \end{equation*}
  choices of $\xi_{k,1},\dots,\xi_{k,r_k}$ and the same number of independent choices of $\xi'_{k,1},\dots,\xi'_{k,r_k}$. This yields
  \begin{align*}
    \sum_{\text{diagonal}} \prod_{k\in\K^+}& \left(\prod_{n=1}^{r_k}a_{\xi_{k,n}}\prod_{m=1}^{s_k}\bar{a}_{\xi'_{k,m}}\right) =\\
    &\prod_{k\in\K^+} \sum_{u_k=1}^{r_k} \sideset{}{'} \sum \left(\frac{r_k !}{\prod_{\alpha=1}^{u_k}m_{k, \alpha} !} \right)^2 \sum_{n_{k,1},\dots,n_{k,u_k}} \prod_{\alpha=1}^{u_k}|a_{\Xi_{k,n_{k,\alpha} } } |^{2m_{k,\alpha}}
  \end{align*}
  where $\sideset{}{'} \sum$ indicates that we sum over all positive integers $m_{k,1},\dots,m_{k,u_k}$ that satisfy $m_{k,1}+\dots+m_{k,u_k}=r_k$. By symmetry we have
  \begin{equation*}
    \sum_{\text{diagonal}} \prod_{k\in\K^+}\left(\prod_{n=1}^{r_k}a_{\xi_{k,n}}\prod_{m=1}^{s_k}\bar{a}_{\xi'_{k,m}}\right) = \prod_{k\in\K^+} \sum_{u_k=1}^{r_k} \sideset{}{'} \sum \left(\frac{r_k !}{\prod_{\alpha=1}^{u_k}m_{k,\alpha} !} \right)^2 \frac{1}{u_k !} \sideset{}{''} \sum \prod_{\alpha=1}^{u_k}|a_{\Xi_{k,n_{k,\alpha} } } |^{2m_{k,\alpha}}
  \end{equation*}
  where $\sideset{}{'} \sum$ is as before and $\sideset{}{''} \sum$ indicates we sum over all $u_k$-tuples of distinct integers $1\leq n_{k,1},\dots,n_{k,u_k}\leq N_k$. Defining
  \begin{equation*}
    A_k(m_{k,1},\dots,m_{k,u_k})=\mu_{E,a}(I_k)^{-r_k} \sideset{}{''} \sum \prod_{\alpha=1}^{u_k}|a_{\Xi_{k,n_{k,\alpha} } } |^{2m_{k,\alpha}}
  \end{equation*}
  we have
  \begin{equation*}
    \begin{split}
      \left(\prod_{k\in\K^+} \mu_{E,a}(I_k)^{-r_k}\right) \sum_{\text{diagonal}} \prod_{k\in\K^+}&\left(\prod_{n=1}^{r_k}a_{\xi_{k,n}}\prod_{m=1}^{s_k}\bar{a}_{\xi'_{k,m}}\right) \\
      &= \prod_{k\in\K^+} \sum_{u_k=1}^{r_k} \sideset{}{'} \sum \left(\frac{r_k !}{\prod_{\alpha=1}^{u_k}m_{k,\alpha} !} \right)^2 \frac{1}{u_k !} A_k(m_{k,1},\dots,m_{k,u_k}).
    \end{split}
  \end{equation*}

  Recall from \eqref{eq: M def} that we have $|a_\xi|^2\leq M$ for all $\xi\in\cE$. Denote by $\sideset{}{'''} \sum$ the sum over \emph{all} $u_k$-tuples of integers $1\leq n_{k,1},\dots,n_{k,u_k}\leq N_k$ and note that
  \begin{align*}
    A_k(m_{k,1},\dots,m_{k,u_k})& \leq\mu_{E,a}(I_k)^{-r_k} M^{r_k-u_k}\sideset{}{''} \sum \prod_{\alpha=1}^{u_k}|a_{\Xi_{k,n_{k,\alpha} } } |^{2} \\
    &\leq \mu_{E,a}(I_k)^{-r_k} M^{r_k-u_k} \sideset{}{'''} \sum \prod_{\alpha=1}^{u_k}|a_{\Xi_{k,n_{k,\alpha} } } |^{2}\\
    &= \left(\frac{M}{\mu_{E,a}(I_k)}\right)^{r_k-u_k}.
  \end{align*}
  Thus for $k\in\K^+$ we have
  \begin{equation*}
    \sum_{u_k=1}^{r_k-1} \sideset{}{'} \sum \left(\frac{r_k !}{\prod_{\alpha=1}^{u_k}m_{k,\alpha} !} \right)^2 \frac{1}{u_k !} A_k(m_{k,1},\dots,m_{k,u_k})\leq C(r_k) \frac{M}{\mu_{E,a}(I_k)} \leq \frac{C(r_k) M}{\delta}.
  \end{equation*}
  On the other hand, since
  \begin{equation*}
    \sideset{}{'''} \sum \prod_{\alpha=1}^{u_k}|a_{\Xi_{k,n_{k,\alpha} } } |^{2}=\left(\sum_{n=1}^{N_k}|a_{\Xi_{k,n } } |^{2}\right)^{u_k} = \mu_{E,a}(I_k)^{u_k},
  \end{equation*}
  we have
  \begin{equation*}
    A_k(\underbrace{1,\dots,1}_{r_k \text{ times}}) = 1-\mu_{E,a}(I_k)^{-r_k}\sum \prod_{\alpha=1}^{r_k}|a_{\Xi_{k,n_{k,\alpha} } } |^{2},
  \end{equation*}
  where the sum runs over all tuples $1\leq n_{k,1},\dots,n_{k,r_k}\leq N_k$ with at least one repeated index. Given an $(r_k -1)$-tuple of integers we can generate at most $r_k^2$ tuples of integers of length $r_k$ with at least one repeated index by repeating one of the indices. Since this generates all such tuples we have, denoting by $\sideset{}{^o}\sum$ the sum over all integers $1\leq \tilde{n}_{k,1},\dots,\tilde{n}_{k,r_k-1}\leq N_k$,
  \begin{equation*}
    \mu_{E,a}(I_k)^{-r_k}\sum \prod_{\alpha=1}^{r_k}|a_{\Xi_{k,n_{k,\alpha} } } |^{2}\leq \frac{M r_k^2}{\mu_{E,a}(I_k)} \mu_{E,a}(I_k)^{1-r_k}\sideset{}{^o}\sum \prod_{\alpha=1}^{r_k-1}|a_{\Xi_{k,\tilde{n}_{k,\alpha} } } |^{2} = \frac{M r_k^2}{\mu_{E,a}(I_k)}\leq \frac{M r_k^2}{\delta}.
  \end{equation*}
  Putting all of this together we see that
  \begin{equation*}
    \sum_{u_k=1}^{r_k} \sideset{}{'} \sum \left(\frac{r_k !}{\prod_{\alpha=1}^{u_k}m_{k,\alpha} !} \right)^2 \frac{1}{u_k !} A_k(m_{k,1},\dots,m_{k,u_k}) = r_k! + O_{r_k}\left(\frac{M}{\delta}\right)
  \end{equation*}
  which yields
  \begin{align*}
    \left(\prod_{k\in\K^+} \mu_{E,a}(I_k)^{-r_k}\right) \sum_{\text{diagonal}} \prod_{k\in\K^+} \left(\prod_{n=1}^{r_k}a_{\xi_{k,n}}\prod_{m=1}^{s_k}\bar{a}_{\xi'_{k,m}}\right) &= \prod_{k\in\K^+} \left(r_k! + O_{r_k}\left(\frac{M^2 }{\delta}\right) \right)\\
    & =\left(\prod_{k\in\K^+} r_k!\right) + O_{B_1}\left(\frac{M}{\delta}\right),
  \end{align*}
  as claimed.
\end{proof}

\begin{proof}[Proof of Claim~\ref{cl: second claim}]
  Trivially we have
  \begin{equation}\label{eq: off diag bound}
    \Abs{\sum_{\text{off-diagonal}} \prod_{k\in\K^+}\left(\prod_{n=1}^{r_k}a_{\xi_{k,n}}\prod_{m=1}^{s_k}\bar{a}_{\xi'_{k,m}}\right)}\leq J M^{\sum_{k\in\K^+}\frac{r_k+s_k}{2}},
  \end{equation}
  where $J$ is the number of off-diagonal relations
  \begin{equation*}
    \sum_{k\in\K^+} (\xi_{k,1}+\cdots+\xi_{k,r_k}-\xi'_{k,1}-\cdots-\xi'_{k,s_k}) =0.
  \end{equation*}
  This is clearly bounded by the number of relations
  \begin{equation}\label{eq: vanishing relation}
    \xi_{1}+\cdots+\xi_{l}=0
  \end{equation}
  where $\xi_{1},\dots,\xi_{l}\in\cE$, $l=\sum_{k\in\K^+}r_k+s_k$ and there is a minimally vanishing sub-sum of \eqref{eq: vanishing relation} of length $l'\geq3$. Suppose that there are $\nu$ pairs of elements in \eqref{eq: vanishing relation} whose sum vanishes. The number of such relations is bounded by
  \begin{equation}\label{eq: bound for vanishing}
    C(l)N^\nu N^{\gamma(l-2\nu)},
  \end{equation}
  where we have applied the fact that $\cE$ satisfies the condition $I(\gamma,B_1)$. Summing \eqref{eq: bound for vanishing} over the values of $\nu$ such that $0\leq2\nu\leq l-3$ gives
  \begin{equation}\label{eq: bound for J}
    J\leq C(l)N^{l/2} N^{3(\gamma-\frac12)}.
  \end{equation}
  Combining \eqref{eq: off diag bound} and \eqref{eq: bound for J} we see that
  \begin{equation*}
    \Abs{\sum_{\text{off-diagonal}} \prod_{k\in\K^+}\left(\prod_{n=1}^{r_k}a_{\xi_{k,n}}\prod_{m=1}^{s_k}\bar{a}_{\xi'_{k,m}}\right)}\leq C(B_1)(MN)^{\sum_{k\in\K^+}\frac{r_k+s_k}{2}} N^{3(\gamma-\frac12)}.
  \end{equation*}
  Applying assumption \eqref{eq: bound on coeff} we see that
  \begin{equation*}
    \prod_{k\in\K^+} \mu_{E,a}(I_k)^{-\frac{r_k+s_k}{2}} \Abs{\sum_{\text{off-diagonal}} \prod_{k\in\K^+}\left(\prod_{n=1}^{r_k}a_{\xi_{k,n}}\prod_{m=1}^{s_k}\bar{a}_{\xi'_{k,m}}\right)}\leq \frac{C(B_1)}{\delta^{B_1/2}} \frac{g(N)^{B_1/2} } {N^{3(\frac12-\gamma)}}
  \end{equation*}
  as claimed.
\end{proof}

\subsection{Proof of Proposition \ref{prop: pass to random}: Bourgain's de-randomisation technique II, an approximately Gaussian field}

Our first step in the proof of Propositon~\ref{prop: pass to random} is the following lemma, which states that the distribution of each $b_k$ is ``nice''.

\begin{lemma}\label{lem: abs cont}
  Suppose that $N\to\infty$ and that the coefficients $a_\xi$ are of class $\cA(g)$, for some slowly growing function $g$. If $\delta>0$ is fixed then, for each $k\in\K^+$, the measure $(b_k)_*\Prr$ is absolutely continuous with respect to Lebesgue measure on the plane, for sufficiently large $E$. In other words, there exists $E_7=E_7(\delta,g)$ such that, if $m$ denotes the Lebesgue measure and $B$ is a measurable set, then for all $E\geq E_7$
  \begin{equation}\label{eq: abs cont}
    m(B)=0\Rightarrow\Prr[b_k\in B]=0.
  \end{equation}
\end{lemma}
\begin{proof}
  We consider $\T^2$ as a subset of $\R^2$, and $b_k$ as a map from $\R^2$ to $\R^2$. Let $J_k$ be the Jacobian of this map. It suffices to show that the set
  \begin{equation*}
    \sigma=\{x\in\T^2: \det J_k(x)=0\}
  \end{equation*}
  has Lebesgue measure zero, since we may apply the inverse function theorem on $\T^2\setminus \sigma$ to see that $\Prr[(\T^2\setminus \sigma)\cap b_k^{-1}(B)]=0$. A routine (but tedious) calculation yields
  \begin{equation*}
    \det J_k(x)=-\frac{4\pi^2}{\mu_{E,a}(I_k)}\im\sum_{\xi,\xi'\in\cE^{(k)}} \xi_1\xi'_2 \overline{a_\xi} a_{\xi'} e(\langle \xi-\xi',x\rangle)
  \end{equation*}
  which means that $x\in \sigma$ if and only if
  \begin{equation}\label{eq: Jacobian =0}
    \sum_{\xi,\xi'\in\cE^{(k)}} (\xi_1\xi'_2- \xi_2\xi'_1) \overline{a_\xi} a_{\xi'} e(\langle \xi-\xi',x\rangle)=0.
  \end{equation}

  It therefore suffices to see that at least one of the coefficients in the expansion \eqref{eq: Jacobian =0} is non-zero. Using the fact that $\mu_{E,a}(I_k) \geq \delta$ for each $k\in\K^+$, and assumption \eqref{eq: bound on coeff} along with the hypothesis that $g$ is a slowly growing function, we see that, for large enough $E$, there exists at least two $\xi\in\cE^{(k)}$ with $a_\xi\neq0$. Let $\xi_k^+$ (respectively $\xi_k^-$) be the element of $\{\xi\in\cE^{(k)}\colon a_\xi\neq0\}$ with the largest (respectively smallest) argument. Then the coefficient of $e(\langle \xi_k^+-\xi_k^-,x\rangle)$ is $((\xi_k^+)_1(\xi_k^-)_2- (\xi_k^+)_2(\xi_k^-)_1) \overline{a_{\xi_k^+}} a_{\xi_k^-}$ which is clearly non-zero. This completes the proof of Lemma~\ref{lem: abs cont}.
\end{proof}

Our next step is to show that the distribution of the quasi-random sequence $\{b_k\}_{k\in\K^+}$ is ``close'' to the distribution of the sequence $\{c_k\}_{k\in\K^+}$, at least in terms of the probability of certain events. We begin with some notation. We define $\kappa=|\K^+|$ and write $Q(z,L)$ for the (closed) cube  centred at $z$ of side-length $L$ in $\R^{2\kappa}$ (which we identify with $\C^\kappa$). Given a parameter $\eta$ we divide $Q(z,L)$ into $\lceil \frac{L}{\eta} \rceil^{2\kappa}$ sub-cubes of side-length at most $\eta$, which we label $Q_\beta^{z,L,\eta}$ for $1\leq\beta\leq\lceil \frac{L}{\eta} \rceil^{2\kappa}$. Here the indexing $\beta$ is arbitrary, and we also arbitrarily assign the boundaries of the sub-cubes in some consistent manner, that will be unimportant for our purposes.

Given another parameter $\Upsilon$ we define
\begin{equation*}
  \Lambda_{\Upsilon,\eta}=\frac \eta \Upsilon \Z^{2\kappa} \cap Q\left(0,\frac\eta 2\right) =\left\{ \left(\frac \eta \Upsilon a_1,\dots,\frac \eta \Upsilon a_{2\kappa}\right) \colon -\frac \Upsilon2\leq a_k\leq\frac \Upsilon2,\quad a_k\in\Z\right\}.
\end{equation*}
We are ready to state the next lemma.

\begin{lemma}[The quasi-distribution of $b_k$ is approximately Gaussian]\label{lem: eps Gaussian}
   Let $0<\epsilon_7,\delta,\eta<1$ and $K,L>1$ be given. Suppose that $N\to\infty$ as $E\to\infty$ and that the coefficients $a_\xi$ are of class $\cA(g)$, for some slowly growing function $g$. There exists $\lambda\in\Lambda_{\frac6{\epsilon_7},\eta}$, $B_2=B_2(K,L,\epsilon_7,\eta)$ and $E_8=E_8(K,L,\epsilon_7,\delta,\eta,\gamma,g)$ such that if $\cE$ satisfies $I(\gamma,B_2)$ for all $E$ and for some fixed $\gamma$, then for all $E\geq E_8$
  \begin{equation}\label{eq: eps gauss big}
    \big|\Pro[c\in Q(\lambda,L)]-\Prr[b\in Q(\lambda,L)]\big|\leq \epsilon_7
  \end{equation}
  and
  \begin{equation}\label{eq: eps gauss small}
    \big|\Pro[c\in Q_\beta^{\lambda,L,\eta}]-\Prr[b\in Q_\beta^{\lambda,L,\eta}]\big|\leq \epsilon_7
  \end{equation}
  for all $\beta$.
\end{lemma}

\begin{proof}
  We let $k_1<\dots<k_\kappa$ be the elements of $\K^+$ and define real-valued quasi-random variables $(X_j)_{j=1}^{2\kappa}$ by $X_{2j-1}=\re(b_{k_j})$ and $X_{2j}=\im(b_{k_j})$. For $t=(t_1,\dots ,t_{2\kappa})\in \R^{2\kappa}$ we define
  \begin{equation*}
    \varphi(t)=\Ee\left[e^{i\langle t,X\rangle }\right]
  \end{equation*}
  where $\langle\cdot,\cdot\rangle$ is the usual inner product in $\R^{2\kappa}$.

  For any (bounded) cube $Q=\prod_{j=1}^{2\kappa}(p_j,q_j)$ (which is a continuity set, by Lemma~\ref{lem: abs cont}) we have, by Levy's inversion formula (see \cite{Bil}*{Theorem 26.2} for the one-dimensional case),
  \begin{equation}\label{eq: Levy}
    \Prr[b\in Q]=\frac1{(2\pi)^{2\kappa}} \lim_{T_1,\dots,T_{2\kappa}\to\infty} \int_{-T_1}^{T_1} \cdots \int_{-T_{2\kappa}}^{T_{2\kappa}}\prod_{j=1}^{2\kappa} \left( \frac{ e^{-it_j p_j} - e^{-it_j q_j} } {it_j} \right) \varphi(t)\,dt_{2\kappa}\dots\,dt_1.
  \end{equation}
  Applying Fubini's theorem (which is allowed because $|e^{i\langle t,X\rangle  }|\leq1$ and $\left|\frac{ e^{-it_j p_j} - e^{-it_j q_j} } {it_j}\right|\leq q_j-p_j$) we get
  \begin{align}\label{eq: Levy inv finite}
    \int_{-T_1}^{T_1} \cdots \int_{-T_{2\kappa}}^{T_{2\kappa}}\prod_{j=1}^{2\kappa} \Big( \frac{ e^{-it_j p_j} - e^{-it_j q_j} } {it_j}& \Big)  \varphi(t)\,dt_{2\kappa}\dots\,dt_1\notag\\
    &=\Ee\left[\prod_{j=1}^{2\kappa}\int_{-T_j}^{T_j} \frac{ e^{it_j(X_j- p_j)} - e^{it_j (X_j-q_j)} } {it_j} \,dt_j \right]\notag\\
    &=\Ee\left[\prod_{j=1}^{2\kappa}\int_{-T_j}^{T_j} \frac{ \sin(t_j(X_j- p_j)) - \sin(t_j (X_j-q_j)) } {t_j} \,dt_j \right],
  \end{align}
  where the last equality follows from parity considerations. We write $\sgn(x)=+1,0$ or $-1$ if $x$ is positive, $0$ or negative respectively, and define $S(T)=\int_0^T\frac{\sin t}t \,dt$ for $T>0$. We re-write \eqref{eq: Levy inv finite} as
  \begin{align}\label{eq: Levy 2}
    \int_{-T_1}^{T_1} \cdots \int_{-T_{2\kappa}}^{T_{2\kappa}}\prod_{j=1}^{2\kappa} &\Big( \frac{ e^{-it_j p_j} - e^{-it_j q_j} } {it_j} \Big)  \varphi(t)\,dt_{2\kappa}\dots\,dt_1\notag\\
    &=\Ee\left[\prod_{j=1}^{2\kappa}2\big(\sgn(X_j- p_j)S(T_j|X_j- p_j|) - \sgn(X_j- q_j)S(T_j|X_j- q_j|)\big) \right].
  \end{align}
  Recall that $S(T)=\frac\pi 2 + O(T^{-1})$ for large $T$ (though the (infinite) integral is not absolutely convergent). Moreover $0\leq S(T)\leq S(\pi)\leq\pi$ for all $T>0$. This implies that the expression inside the quasi-expectation in \eqref{eq: Levy 2} is bounded, and so we may apply dominated convergence to see that, taking into account \eqref{eq: Levy}, for any $T_0>0$
  \begin{align}\label{eq: long expr Levy}
    \Prr[b\in Q]-\frac1{(2\pi)^{2\kappa}}& \int_{-T_0}^{T_0} \cdots \int_{-T_{0}}^{T_{0}}\prod_{j=1}^{2\kappa} \left( \frac{ e^{-it_j p_j} - e^{-it_j q_j} } {it_j} \right) \varphi(t) \,dt_{2\kappa}\dots\,dt_1\notag\\
    &= \frac1{(2\pi)^{2\kappa}} \Ee\bigg[\prod_{j=1}^{2\kappa}2\bigg(\sgn(X_j- p_j)\left(\frac\pi 2-S(T_0|X_j- p_j|)\right)\\
    &\qquad\qquad\qquad\qquad\qquad - \sgn(X_j- q_j)\left(\frac\pi 2-S(T_0|X_j- q_j|)\right)\bigg) \bigg]\notag.
  \end{align}

  \begin{claim}\label{cl: choose lambda}
    If $\Upsilon\geq1$ then there exists $\lambda\in\Lambda_{\Upsilon,\eta}$ such that
    \begin{equation*}
      \Prr\left[b\in\cup_{\beta}\left(\partial Q_\beta^{\lambda,L,\eta}+Q\big(0,\frac{\eta}{2\Upsilon}\big)\right)\right]\leq \frac1\Upsilon.
    \end{equation*}
  \end{claim}
  We will prove this claim at the end. Choosing $\Upsilon=\frac6{\epsilon_7}$, fixing a $\beta$, choosing $\lambda$ to satisfy the claim and applying \eqref{eq: long expr Levy} to $Q_\beta^{\lambda,L,\eta}$ we see that
  \begin{align}\label{eq: Levy difference fin inf}
    \Abs{\Prr[b\in Q_\beta^{\lambda,L,\eta}]-\frac1{(2\pi)^{2\kappa}}\int_{-T_0}^{T_0} \cdots \int_{-T_{0}}^{T_{0}}\prod_{j=1}^{2\kappa} \left( \frac{ e^{-it_j p_j} - e^{-it_j q_j} } {it_j} \right) \varphi(t) \,dt_{2\kappa}\dots\,dt_1}\\
    \leq \left(\frac {C\Upsilon}{T_0 \eta}\right)^{2\kappa} + \frac1\Upsilon;\notag
  \end{align}
  where $p_j$ and $q_j$ are now defined in the obvious way; here we have used the fact that
  \begin{equation*}
    2\bigg(\Abs{\frac\pi 2-S(T_0|X_j- p_j|)}+\Abs{\frac\pi 2-S(T_0|X_j- q_j|)}\bigg)\leq2\pi
  \end{equation*}
  and the fact that  on the complement of the event $\{b\in\cup_{\beta}\left(\partial Q_\beta^{\lambda,L,\eta}+Q\big(0,\frac{\eta}{2\Upsilon}\big)\right)\}$ we have $|X_j- p_j|,|X_j- q_j|\geq \frac{\eta}{2\Upsilon}$. Choosing $T_0$ large (depending on $\epsilon_7$ and $\eta$), and taking into account our choice of $\Upsilon$, we see that the right hand side of \eqref{eq: Levy difference fin inf} is at most $\frac{\epsilon_7}3$, and this is uniform in $\beta$.

  Similar computations yield
  \begin{equation*}
    \Abs{\Pro[c\in Q_\beta^{\lambda,L,\eta}]-\frac1{(2\pi)^{2\kappa}}\int_{-T_0}^{T_0} \cdots \int_{-T_{0}}^{T_{0}}\prod_{j=1}^{2\kappa} \left( \frac{ e^{-it_j p_j} - e^{-it_j q_j} } {it_j} \right) e^{- \frac {|t|^2} {4} } \,dt_{2\kappa}\dots\,dt_1}\leq \frac{\epsilon_7}3.
  \end{equation*}
  Thus to prove \eqref{eq: eps gauss small} it remains only to show that
  \begin{equation}\label{eq: 3eps}
    \Abs{\frac1{(2\pi)^{2\kappa}}\int_{-T_0}^{T_0} \cdots \int_{-T_{0}}^{T_{0}}\prod_{j=1}^{2\kappa} \left( \frac{ e^{-it_j p_j} - e^{-it_j q_j} } {it_j} \right) (\varphi(t) - e^{- \frac {|t|^2} {4} }) \,dt_{2\kappa}\dots\,dt_1}\leq \frac{\epsilon_7}3.
  \end{equation}
  The left-hand side of \eqref{eq: 3eps} is clearly bounded by
  \begin{equation*}
    \left(\frac{T_0 \eta}{\pi}\right)^{2\kappa} \max_{t\in[-T_0,T_0]^{2\kappa}}\Abs{\varphi(t) - e^{- \frac {|t|^2} {4}}}.
  \end{equation*}
  But since the difference $\Abs{\varphi(t) - e^{- \frac {|t|^2} {4}}}$ on a compact interval can be estimated by the difference of moments, we see that the required conclusion is a direct consequence of Proposition~\ref{prop:eps  Gauss mom}, by choosing an appropriately large $B_1$ (depending on $\epsilon_7,\eta$ and $K$). The proof that
  \begin{equation*}
    \big|\Pro[c\in Q(\lambda,L)]-\Prr[b\in Q(\lambda,L)]\big|\leq \epsilon_7
  \end{equation*}
  is similar, and omitted.
\end{proof}

\begin{proof}[Proof of Claim~\ref{cl: choose lambda}]
  Note that
  \begin{equation*}
    \sum_{\lambda\in\Lambda_{\Upsilon,\eta}}\Prr\left[ b\in\cup_{\beta}\left(\partial Q_\beta^{\lambda,L,\eta}+Q\big(0,\frac{\eta}{2\Upsilon}\big)\right) \right]=\Prr\left[ b\in Q\big(0,L+\frac{\eta}{2}\big) \right]\leq 1,
  \end{equation*}
  and since $|\Lambda_{\Upsilon,\eta}|\geq \Upsilon^{2\kappa}$ we see that there exists at least one $\lambda$ such that
  \begin{equation*}
    \Prr\left[ b\in\cup_{\beta}\left(\partial Q_\beta^{\lambda,L,\eta}+Q\big(0,\frac{\eta}{2\Upsilon}\big)\right) \right]\leq \Upsilon^{-2\kappa}\leq \Upsilon^{-1},
  \end{equation*}
  as claimed.
\end{proof}

\begin{proof}[Proof of Proposition~\ref{prop: pass to random}]
  We begin by constructing $\tau$, we then will use this to construct a $\Psi$, and finally verify that they satisfies properties (i)-(iv). Let $L=2\sqrt{2\log\frac{K}{\epsilon_3}}$, $\eta=\frac{\epsilon_1}{4\sqrt{2}\pi R K}$ and $0<\epsilon_7<1$ (to be specified). Let $Q(\lambda,L)$ be the cube constructed in Lemma~\ref{lem: eps Gaussian} and $Q_\beta^{\lambda,L,\eta}$ its corresponding decomposition. For each $\beta$ we consider the event
  \begin{equation*}
    A_\beta=\{c\in Q_\beta^{\lambda,L,\eta}\}
  \end{equation*}
  and the quasi-event
  \begin{equation*}
    \mathcal{Q}_\beta=\{b\in Q_\beta^{\lambda,L,\eta}\}.
  \end{equation*}
  (Note that by \eqref{eq: abs cont} and the fact that $c$ is a Gaussian vector, we don't need to worry about the boundary of the cubes.) By \eqref{eq: eps gauss small} we have
  \begin{equation}\label{eq: AB close}
    \big|\Pro[A_\beta]-\Prr[\mathcal{Q}_\beta]\big|\leq \epsilon_7
  \end{equation}
  for large enough $E$. Now since $\Pro[A_\beta]\geq e^{-2KL^2}\eta^{2K}\pi^{-K}$, we have
  \begin{equation*}
    \big|\Pro[A_\beta]-\Prr[\mathcal{Q}_\beta]\big|\leq \epsilon_7< \epsilon_3 \Pro[A_\beta]
  \end{equation*}
  provided that
  \begin{equation}\label{eq: choice eps5}
    \epsilon_7 < \frac{\eta^{2K}\epsilon_3 e^{-2KL^2}}{\pi^{K}}.
  \end{equation}

  Now, by \eqref{eq: abs cont} and the fact that $c$ is a Gaussian vector, for each $\beta$ we may find $A_\beta'\subset A_\beta$ and $\mathcal{Q}_\beta'\subset \mathcal{Q}_\beta$
  \begin{equation}\label{eq: proA alpha estimate}
    \Pro[A_\beta']=\Prr[\mathcal{Q}_\beta']>\left(1-\epsilon_3\right)\Pro[A_\beta].
  \end{equation}
  Furthermore, by\footnote{Actually, combining \cite{Roy}*{Theorem 15.4} with \cite{Roy}*{Proposition 15.3} we see that there exists a measure preserving map $\tau_\beta'\colon A_\beta'\to [0,1]$, where the measure on $[0,1]$ is the Lebesgue measure. However, by modifying the proof of \cite{Roy}*{Theorem 15.4} appropriately, it is easy to construct a measure preserving map $\tau_\beta''[0,1]\to \mathcal{Q}_\beta'$, and the desired map is given by composing these two mappings.} \cite{Roy}*{Theorem 15.4}, we may find a measure preserving map $\tau_\beta\colon A_\beta'\to \mathcal{Q}_\beta'$, that is, $(\tau_\beta)_*\Pro=\Prr$ on $A_\beta'$. We define $\Omega'=\cup_\beta A_\beta'$ and $\tau$ by $\tau(\omega)=\tau_\beta(\omega)$ for $\omega\in A_\beta'$.

  Now for $\omega\in\Omega'$ we define
  \begin{equation}\label{eq: Psi def}
    \begin{split}
      \Psi_\omega(y)&=\psi_{\tau(\omega)}(y)+\phi_{\tau(\omega)}(y)-\Phi_{\omega}(y)\\
      &=\psi_{\tau(\omega)}(y)+\sum_{k\in\K} \mu(I_k)^{\frac12} \{b_k(\tau(\omega))-c_k(\omega)\} e(\langle R\zeta^{(k)},y\rangle)
    \end{split}
  \end{equation}
  and put $\Psi_\omega(y)=0$ otherwise.

  We begin by checking (i). Note that for all $k$
  \begin{equation*}
    \|e(\langle R\zeta^{(k)},\cdot\rangle)\|_{\mathcal C^1}\leq2\pi R
  \end{equation*}
  (since $R>1$) so that, for all $\omega\in\Omega'$, we have
  \begin{align*}
    \|\Psi_\omega\|_{\mathcal C^1}&\leq\|\psi_{\tau(\omega)}\|_{\mathcal C^1}+\sum_{k\in\K} |b_k(\tau(\omega))-c_k(\omega)| \|e(\langle R\zeta^{(k)},\cdot\rangle)\|_{\mathcal C^1}\\
    &\leq\epsilon_1+ 2\pi R\sum_{k\in\K} |b_k(\tau(\omega))-c_k(\omega)|,
  \end{align*}
  where we have applied Proposition~\ref{prop: gathering}. Now each $\omega\in\Omega'$ is in $A_\beta'$ for some $\beta$, and so the vectors $c(\omega)$ and $b(\tau(\omega))$ are in the same cube $Q_\beta^{\lambda,L,\eta}$, which means that $|b_k(\tau(\omega))-c_k(\omega)|\leq\sqrt{2}\eta$. We therefore have
  \begin{equation*}
    \|\Psi_\omega\|_{\mathcal C^1}\leq\epsilon_1+ 4\sqrt{2}\pi RK\eta=2\epsilon_1
  \end{equation*}
  for all $\omega\in\Omega'$ by our choice of $\eta$, and since $\Psi_\omega(y)=0$ for $\omega\notin\Omega'$ we see that (i) holds.

  Now for $\omega\in A_\beta'$ we have $\Phi_\omega+\Psi_\omega=\phi_{\tau_\beta(\omega)}+\psi_{\tau_\beta(\omega)}$ by \eqref{eq: Psi def} and so, by Proposition~\ref{prop: gathering}, (ii) holds on $\Omega'$. Otherwise, $\Psi_\omega(y)=0$ and a routine calculation shows that $\Delta\Phi_\omega=-4\pi^2R^2\Phi_\omega$.

  To show (iii) , we first note that
  \begin{equation*}
    \Omega\setminus\Omega'= \{c\notin Q(\lambda,L) \}\cup(\cup_\beta A_\beta\setminus A_\beta')
  \end{equation*}
  and
  \begin{equation*}
    \Pro[c\notin Q(\lambda,L)]\leq\Pro\left[\cup_k \left\{|c_k|>\frac{L-\eta}2\right\}\right]\leq Ke^{-\frac{L^2}8}=\epsilon_3
  \end{equation*}
  by our choice of $L$, while, by \eqref{eq: proA alpha estimate},
  \begin{equation*}
    \Pro[\cup_\beta A_\beta\setminus A_\beta']\leq\sum_\beta\Pro[A_\beta\setminus A_\beta']\leq\epsilon_3\sum_\beta\Pro[A_\beta]\leq\epsilon_3.
  \end{equation*}
  We also have
  \begin{equation*}
    \T^2\setminus\tau(\Omega')= \{b\notin Q(\lambda,L) \}\cup(\cup_\beta \mathcal{Q}_\beta\setminus \mathcal{Q}_\beta')
  \end{equation*}
  and, by \eqref{eq: eps gauss big},
  \begin{equation*}
    \Prr[b\notin Q(\lambda,L)]\leq\epsilon_7+\Pro[c\notin Q(\lambda,L)]\leq\epsilon_7+\epsilon_3\leq 2\epsilon_3
  \end{equation*}
  by \eqref{eq: choice eps5} , while, by \eqref{eq: proA alpha estimate} and \eqref{eq: AB close},
  \begin{equation*}
    \Prr[\cup_\beta \mathcal{Q}_\beta\setminus \mathcal{Q}_\beta']\leq\sum_\beta\Prr[\mathcal{Q}_\beta\setminus \mathcal{Q}_\beta']\leq\sum_\beta\left((1+\epsilon_3)\Pro[A_\beta]-\Pro[A_\beta']\right)\leq 2\epsilon_3.
  \end{equation*}
  This shows (iii).

  It remains to verify (iv). Note first that by Courant's Theorem we have $\cN_{\Phi_\omega}=O(R^2)$ and $\cN_{\phi_x+\psi_x}=O(R^2)$ (deterministically, i.e., for all $\omega$ and $x$). We therefore have
  \begin{align*}
    \E[\cN_{\Phi+\Psi}]&=\int_{\Omega'}\cN_{\Phi_\omega+\Psi_\omega}d\Pro(\omega)+\int_{\Omega\setminus\left(\Omega'\right)}\cN_{\Phi_\omega}d\Pro[\omega]\\
    &=\sum_\beta\int_{A_\beta'}\cN_{\phi_{\tau_\beta(\omega)}+\psi_{\tau_\beta(\omega)}}d\Pro(\omega)+O(R^2)\Pro[\Omega\setminus\left(\Omega'\right)]\\
    &=\sum_\beta\int_{\mathcal{Q}_\beta'}\cN_{\phi_{x}+\psi_{x}}d\Prr(x)+O(R^2)\Pro[\Omega\setminus\left(\Omega'\right)]\\
    &=\Ee[\cN_{\phi+\psi}]+O(R^2)\left(\Prr[\T^2\setminus\tau(\Omega')]+\Pro[\Omega\setminus\left(\Omega'\right)]\right).
  \end{align*}
  The estimates for these probabilities in (iii) give us
  \begin{equation*}
    \E[\cN_{\Phi+\Psi}]=\Ee[\cN_{\phi+\psi}]+O(\epsilon_3 R^2),
  \end{equation*}
  as desired.
\end{proof}

\section{Applying Nazarov and Sodin's results}\label{sec: Naz Sod}

\subsection{Proof of Proposition \ref{prop: Psi negligible}}

The proof of this  proposition requires two lemmas, both coming from \cite{NaSo2}. We begin with some notation. $Q_{+t}$ denotes a $t$-neighbourhood of the set $Q\subseteq\R^2$. We say that a function $r\colon U\times U \to\R$ belongs to the class $C^{k,k}(U\times U)$ if all of the partial derivatives $\partial_y^\alpha\partial_{y'}^\beta \,r(y,y')$ with $|\alpha|,|\beta|\leq k$ (taken in any order) exist and are continuous.
\begin{lemma}[\cite{NaSo2}*{Lemma 7}]\label{lem: Sodin f and grad not simult small}
  Suppose that $Q\subseteq \R^2$ is an open square, $U$ is an open neighbourhood of $\overline{Q_{+1/2}}$ and $h\colon U\to\R$ is a continuous Gaussian field with covariance kernel $r(y,y')$. Suppose further that $r(y,y)=1$, that $r\in C^{2,2}(U\times U)$, that
  \begin{equation*}
    \max_{|\alpha|\leq2}\max_{y\in\overline{Q_{+1/2}}}\Abs{\partial_y^\alpha\partial_{y'}^\alpha \,r(y,y')|_{y'=y}}\leq \Theta,
  \end{equation*}
  and that the field $h$ is non-degenerate and satisfies
  \begin{equation*}
    \det \Sigma_y\geq\theta>0
  \end{equation*}
  where $\Sigma_y$ is the covariance matrix of the Gaussian random vector $\nabla h(y)$, that is, the matrix with the entries $\Sigma_y(i, j) = \partial_{y_i}\partial_{y'_j}\,r(y, y')|_{y'=y}$. Given $\epsilon_4>0$ there exists $\tau_0=\tau_0(\epsilon_4,\theta,\Theta,Q)$ such that
  \begin{equation*}
    \Pro[\min_{y\in Q} \max \{ |h(y)|, |\nabla h(y)|\}\leq \tau]\leq \epsilon_4
  \end{equation*}
  for all $\tau\leq\tau_0$.
\end{lemma}
\begin{rem}
  This is in fact a slight modification of \cite{NaSo2}*{Lemma 7}.
\end{rem}

\begin{lemma}[\cite{NaSo2}*{Lemma 3}]\label{lem: Sodin components}
  Let $Q\subseteq \R^2$ be an open square and $h_1\in\mathcal C^1(Q)$. Suppose that there exist $\tau_1,\tau_2>0$ such that for all $x\in Q$ we have either $|h_1(y)|>\tau_1$ or $|\nabla h_1(y)|>\tau_2$. Let $h_2\in \mathcal C(Q)$ with $\sup_{y\in Q} |h_2(y)|<\tau_1$. Then each component $\Gamma$ of the zero set $\mathcal Z(h_1)$ with $d(\Gamma,\partial Q)>\frac{\tau_1}{\tau_2}$ generates a component $\tilde{\Gamma}$ of the zero set $\mathcal Z(h_1+h_2)$ with $\tilde{\Gamma}\subseteq\Gamma_{+\frac{\tau_1}{\tau_2}}$. Moreover different components of $\mathcal Z(h_1)$ generate different components of $\mathcal Z(h_1+h_2)$.
\end{lemma}

\begin{proof}[Proof of Proposition~\ref{prop: Psi negligible}]
  We first note that since we are interested in the nodal count, we may replace $\Phi$ and $\Psi$ by $(\sum_{k\in\K}\mu_{E,a}(I_k))^{-1/2}\Phi$ and $(\sum_{k\in\K}\mu_{E,a}(I_k))^{-1/2}\Psi$ respectively. We define $Q=(-1/2,1/2)^2$, note that $Q_{+1/2}=(-1,1)^2$ and think of (this modified) $\Phi$ as a Gaussian field on any $U$, an open neighbourhood of $\overline{Q_{+1/2}}$. In order to apply Lemma~\ref{lem: Sodin f and grad not simult small} we first note that this $\Phi$ has co-variance function
  \begin{equation*}
    r(y,y')=\sum_{k\in\K}\mu_k e(\langle R\zeta^{(k)},y-y'\rangle),
  \end{equation*}
  where $\mu_k=\frac{\mu_{E,a}(I_k)}{\sum_{k\in\K}\mu_{E,a}(I_k)}$. Note that we manifestly have $r(y,y)=\sum_{k\in\K}\mu_k=1$ and $r\in C^{2,2}(U\times U)$, while a routine computation shows that
  \begin{equation*}
    \partial_y^\alpha\partial_{y'}^\alpha \,r(y,y')=(4\pi^2R^2)^{|\alpha|}\sum_{k\in\K}\mu_k (\zeta_1^{(k)})^{2\alpha_1}(\zeta_2^{(k)})^{2\alpha_1} e(\langle R\zeta^{(k)},y-y'\rangle).
  \end{equation*}
  This implies that
  \begin{equation*}
    \max_{|\alpha|\leq2}\max_{y\in\overline{Q_{+1/2}}}\Abs{\partial_y^\alpha\partial_{y'}^\alpha \,r(y,y')|_{y'=y}}\leq 16 \pi^4 R^4.
  \end{equation*}

  We further compute that
  \begin{equation*}
    \partial_{y_i}\partial_{y'_j}\,r(y, y')=4\pi^2R^2 \sum_{k\in\K} \mu_k \zeta_i^{(k)} \zeta_j^{(k)} e(\langle R\zeta^{(k)},y-y'\rangle)
  \end{equation*}
  which implies that
  \begin{align*}
    \det \Sigma_y&=16\pi^4R^4\sum_{k,k'\in\K} \mu_k \mu_{k'} (\zeta_1^{(k)})^2 (\zeta_2^{(k')})^2-\zeta_1^{(k)}\zeta_1^{(k')}\zeta_2^{(k)}\zeta_2^{(k')}\\
    &=8\pi^4R^4\sum_{k,k'\in\K} \mu_k \mu_{k'}\sin^2(\phi_k-\phi_{k'})
  \end{align*}
  where $\zeta^{(k)}=e^{i\phi_k}$. This implies that $\det \Sigma_y=0$ if and only if $\K^+$ is a singleton.

  Define the event
  \begin{equation*}
    \F_{\tau}=\{\min_{y\in [-1,1]^2} \max \{ |\Phi(y)|, |\nabla \Phi(y)|\} > \tau\}
  \end{equation*}
  and consider separately two cases. First we assume that there are two distinct points $\zeta$ and $\zeta'$ in the support of the limiting measure $\mu$ with $\zeta\neq-\zeta'$. Then there exist $\rho(\mu)$ and $c(\mu)$ such that, if $\zeta=e^{i\phi}$ and $\zeta'=e^{i\phi'}$ and $I$ (respectively $I'$) denotes the arc centred at $\zeta$ (respectively $\zeta'$) of length $\rho/2$, we have $|\phi-\phi'|,|\pi-\phi+\phi'|>\rho$ and $\mu(I),\mu(I')>c$. In this case we have
  \begin{equation*}
    \det \Sigma_y=8\pi^4R^4\sum_{k,k'\in\K} \mu_k \mu_{k'}\sin^2(\phi_k-\phi_{k'})\gtrsim \sum_{\substack {k\in\K \\ I_k\cap I\neq \varnothing}} \sum_{\substack {k'\in\K \\ I_k\cap I'\neq \varnothing}} \mu_k \mu_{k'}\sin^2(\phi_k-\phi_{k'})
  \end{equation*}
  and note that for such $k,k'$ we have $|\phi_k-\phi_{k'}|,|\pi-\phi_k+\phi_{k'}|\geq \frac\rho 2 -\frac{2\pi}K\geq\frac\rho 4$ for $K\geq\frac{8\pi}\rho$. This implies that
  \begin{equation*}
    \det \Sigma_y \gtrsim \rho^2\sum_{\substack {k\in\K \\ I_k\cap I\neq \varnothing}} \mu_k \sum_{\substack {k'\in\K \\ I_k\cap I'\neq \varnothing}}  \mu_{k'}\geq \rho^2\sum_{\substack {k\in\K \\ I_k\cap I\neq \varnothing}} \mu_{E,a}(I_k) \sum_{\substack {k'\in\K \\ I_k\cap I'\neq \varnothing}}  \mu_{E,a}(I_k'),
  \end{equation*}
  and since
  \begin{equation*}
     \sum_{\substack {k\in\K \\ I_k\cap I\neq \varnothing}} \mu_{E,a}(I_k)\geq \mu_{E,a}(I)-\delta(\frac{\pi\rho}K +2)\geq \frac c2
  \end{equation*}
  for sufficiently small $\delta$ and sufficiently large $E$, this means that
  \begin{equation*}
    \det \Sigma_y \gtrsim \rho^2c^2.
  \end{equation*}
  Lemma~\ref{lem: Sodin f and grad not simult small} therefore implies that, by choosing $\tau$ sufficiently small (depending on $\epsilon_4,R$ and $\mu$) we have $\Pro[\Omega\setminus \F_{\tau}]\leq \epsilon_4$.

  On the other hand, if the support of $\mu$ consists of two antipodal points, then for a large enough $E$ (depending on $K$ and $\delta$) we have $\K^+=\{k_0\}$. (Strictly speaking, if it happens that the support of $\mu$ lies on endpoints of intervals $I_k$, it might be the case that $\K^+$ contains two elements. In this case, we replace $K$ by $K+1$.) We can then represent $\Phi$ as
  \begin{equation*}
    \Phi(y)=2\mu_{k_0}^{1/2}\re (c_{k_0} e(\langle R\zeta^{(k_0)},y\rangle))=2\mu_{k_0}^{1/2}X \cos(2\pi \langle R\zeta^{(k_0)},y\rangle)+ \alpha),
  \end{equation*}
  where $X$ and $\alpha$ are independent random variables, $X^2$ has an exponential distribution, and $\alpha$ is uniformly distributed on $[0,2\pi]$. A routine computation shows that
  \begin{equation*}
    \nabla\Phi(y)=2\mu_{k_0}^{1/2}X 2\pi R \sin(2\pi \langle R\zeta^{(k_0)},y\rangle)+ \alpha)\zeta^{(k_0)},
  \end{equation*}
  where we are treating $\zeta^{(k_0)}$ as a vector in $\R^2$. This yields
  \begin{align*}
    \min_{y\in [-1,1]^2} \max \{ |\Phi(y)|, |\nabla \Phi(y)|\} &= 2\mu_{k_0}^{1/2}X\min_{t\in [0,2\pi]} \max \{ |\cos t|, 2\pi R|\sin t|\}\\
    &=2\mu_{k_0}^{1/2}X\left(1-\frac1{8\pi^2 R^2}+O\left(\frac1{R^4}\right)\right),
  \end{align*}
  from which we easily deduce that $\Pro[\Omega\setminus \F_{\tau}]\leq \epsilon_4$ for a sufficiently small $\tau$ in this case too. We now fix such a $\tau$.

  By Courant's Theorem we have $\cN_{\Phi+\Psi} = O(R^2)$ and $\cN_{\Phi} = O(R^2)$ which implies that $\int_{\Omega\setminus \F_\tau} \cN_{\Phi+\Psi}d\Pro = O(\epsilon_4 R^2)$ and $\int_{\Omega\setminus \F_\tau}\cN_{\Phi}d\Pro = O(\epsilon_4 R^2)$.

  We write $Q_R^-=[-\frac12+\frac1R,\frac12-\frac1R]^2$ and $Q_R^+=[-\frac12-\frac1R,\frac12+\frac1R]^2$. We have $\|\Psi\|_{\mathcal C^1}\leq 2\epsilon_1$ and we choose $\epsilon_1=\frac{\tau}{4R}$. On the event $\F_\tau$, applying Lemma~\ref{lem: Sodin components} with $h_1=\Phi$, $h_2=\Psi$, $\tau_1=2\epsilon_1$ and $\tau_2=\tau$, we see that (almost surely)
  \begin{equation*}
    \begin{split}
      \cN_{\Phi+\Psi}&= \text{Number of components of }\mathcal Z(\Phi+\Psi)\text{ contained in }Q\\
      &\geq \text{Number of components } \Gamma \text{ of } \mathcal Z(\Phi) \text{ contained in } Q \text{ with } d(\Gamma,\partial Q) > \frac{1}{2R}\\
      &\geq \text{Number of components } \Gamma \text{ of } \mathcal Z(\Phi) \text{ contained in } Q \text{ with } d(\Gamma,\partial Q) > \frac{1}{R}.
    \end{split}
  \end{equation*}

  Furthermore, since $\|\Psi\|_{\mathcal C^1}\leq 2\epsilon_1$ we see that, on the event $\F_\tau$, we have also
  \begin{equation*}
    \min_{y\in[-1,1]^2} \max\{ |(\Phi+\Psi)(y)| , |\nabla(\Phi+\Psi)(y)| \} > \frac{\tau}{2}.
  \end{equation*}
  Applying once more Lemma~\ref{lem: Sodin components} with $h_1=\Phi+\Psi$, $h_2=-\Psi$, $\tau_1=2\epsilon_1$ and $\tau_2=\frac\tau2$, we see that, (alomost surely) on the event $\F_\tau$,
  \begin{equation*}
    \begin{split}
      \cN_{\Phi+\Psi} &= \text{Number of components } \Gamma \text{ of } \mathcal Z(\Phi+\Psi) \text{ contained in } Q_R^+ \text{ with } d(\Gamma,\partial Q_R) > \frac{1}{R}\\
      &\leq \text{Number of components of } \mathcal Z(\Phi) \text{ contained in } Q_R^+.
    \end{split}
  \end{equation*}

  Combining these estimates we see that, on the event $\F_\tau$, we have
  \begin{equation*}
    \Abs{ \cN_{\Phi+\Psi} - \cN_\Phi } \leq \text{Number of components of } \mathcal Z(\Phi) \text{ contained in } Q_R^+ \text{ that intersect } Q_R^+\setminus Q_R^-
  \end{equation*}
  which allows us to conclude that
  \begin{align*}
    \big| \E[ \cN_{\Phi+\Psi} &- \cN_\Phi ] \big|\\
    &\leq \int_{\F_\tau} \Abs{ \cN_{\Phi+\Psi} - \cN_\Phi } d\Pro + \int_{\Omega\setminus \F_\tau} \cN_{\Phi+\Psi} d\Pro + \int_{\Omega\setminus \F_\tau} \cN_{\Phi} d\Pro\\
    &\leq \E\left[\# \text{components of } \mathcal Z(\Phi)\text{ contained in } Q_R^+ \text{ that intersect } Q_R^+\setminus Q_R^- \right] + O(\epsilon_4 R^2).
  \end{align*}
  We may compute the expected number of intersections of $\mathcal Z(\Phi)$ with $\partial Q_R^-$ using the Kac-Rice formula; this is at most $O(R)$. Each nodal domain of $\Phi$ has area at least $\frac{c}{R^2}$, and the area of $Q_R^+\setminus Q_R^-$ is at most $\frac{C}{R}$. Thus the number of components of $\mathcal Z(\Phi)$ contained in $Q_R^+\setminus Q_R^-$ is at most $O(R)$. Since every component of $\mathcal Z(\Phi)$ contained in $Q_R^+$ that intersects $Q_R^+\setminus Q_R^-$ is either contained in $Q_R^+\setminus Q_R^-$ or intersects $\partial Q_R^-$, we conclude that
  \begin{equation*}
    \big| \E[ \cN_{\Phi+\Psi} - \cN_\Phi ] \big| = O(R + \epsilon_4 R^2)
  \end{equation*}
  as claimed.
\end{proof}

\subsection{Proof of Proposition \ref{prop: NS constant}}

First, recalling the definition of $\mu_k$ from the previous subsection, define a Gaussian field
\begin{align*}
  h_\omega\colon&\R^2 \to \R\\
  &y \mapsto \sum_{k\in\K} \mu_k^{\frac12} c_k(\omega) e\langle \zeta^{(k)},y\rangle
\end{align*}
with corresponding spectral measure $\mu_K=\sum_{k\in\K} \mu_k \delta_{\zeta^{(k)}}$. Since, for $y\in[-1,1]^2$ we have
\begin{equation*}
  \Phi(y)=h(Ry)\sum_{k\in\K}\mu_{E,a}(I_k),
\end{equation*}
it is clear that $\cN_\Phi=\cN_{h}(0,R)$ (recall that $\cN_{f}(x,R)$ denotes the number of nodal domains of the function $f$ contained in the open square centred at $x$ of sidelength $R$). Moreover, \cite{NaSo2}*{Theorem 1} or rather the more precise version \cite{KW}*{Proposition 3.5}, implies that for large $R$
\begin{equation*}
  \E[\cN_\Phi]= c_{NS}(\mu_K)R^2 +O(R)
\end{equation*}
where the constant in the term $O(R)$ is absolute (that is, independent of all of the other parameters)

It remains to see that $c_{NS}(\mu_K)$ is close to $c_{NS}(\mu)$. First note that, by Kurlberg and Wigman \cite{KW}*{Theorem 3.1}, the map $\mu\mapsto c_{NS}(\mu)$ is continuous with respect to the weak$^*$ topology, so it is enough to see that $c_{NS}(\mu_K)$ is close to $c_{NS}(\mu_{E,a})$. This topology may be metrized by the Prokhorov metric, and we claim that the set of all probability measures on $\Sone$, which we denote $\mathcal P(\Sone)$, is compact. Indeed, applying Prokhorov's theorem \cite{Bil2}*{Theorem 5.1}, it is enough to see that $\mathcal P(\Sone)$ is tight. But since $\Sone$ is compact this is trivial, so $\mathcal P(\Sone)$ is indeed compact.

Now since a continuous function on a compact space is uniformly continuous, it suffices to show that, given $\eta>0$ there exists $K_2(\eta)$ such that
\begin{equation*}
  d(\mu_K,\mu_{E,a})\leq \eta
\end{equation*}
for all $K\geq K_2$, where $d$ denotes the Prokhorov metric. Recall that this metric is defined by
\begin{align*}
  d(\mu_K,\mu_{E,a}) = \inf\{\eta>0 \colon \mu_K(V)& \leq \mu_{E,a}(V_{+\eta}) + \eta \text{ and }\\
  & \mu_{E,a}(V) \leq \mu_K(V_{+\eta}) + \eta \text{ for all Borel } V\subseteq\Sone \},
\end{align*}
where $V_{+\eta}$ denotes the $\eta$-neighbourhood of the set $V$.

Now the intervals $I_k$ have been chosen so that each has length at most $\frac CK$. Thus, if $V\cap I_k\neq\varnothing$ then $I_k\subseteq V_{+C/K}$. This implies that
\begin{equation*}
  \mu_K(V) \leq \sum_{k:V\cap I_k\neq\varnothing} \mu_K(I_k) = \frac{\sum_{k:V\cap I_k\neq\varnothing} \mu_{E,a}(I_k)}{\sum_{k\in\K}\mu_{E,a}(I_k)} \leq \frac{\mu_{E,a}(V_{+C/K})}{\sum_{k\in\K}\mu_{E,a}(I_k)}
\end{equation*}
and, since
\begin{equation*}
  \sum_{k\in\K}\mu_{E,a}(I_k)=1-\sum_{k\notin\K}\mu_{E,a}(I_k)\geq 1-\delta K,
\end{equation*}
we see that for $\delta\leq K^{-2}$ we have
\begin{equation*}
  \mu_K(V) \leq \mu_{E,a}(V_{+C/K})+\frac{C}{K}.
\end{equation*}
Similarly we have
\begin{equation*}
  \mu_{E,a}(V)\leq \left(\sum_{k\in\K}\mu_{E,a}(I_k)\right) \mu_K(V_{+C/K})\leq \mu_K(V_{+C/K}).
\end{equation*}
This implies that $d(\mu_K,\mu_{E,a})\leq \frac CK$, which completes the proof of Proposition~\ref{prop: NS constant}.

\end{document}